\documentclass[12pt,reqno]{amsart}
\usepackage{amsmath}
\usepackage{amssymb}

\usepackage{waveletnotation}

\newcommand{\Wop}{\mathcal{W}}
\newcommand{\I}{\mathcal{I}}

\newtheorem{lemma}{Lemma}
\newtheorem{prop}[lemma]{Proposition}
\newtheorem{cor}[lemma]{Corollary}
\newtheorem{theorem}[lemma]{Theorem}

\begin{document}
\title{Nonhomogeneous Wavelet Systems in High Dimensions}

\author{Bin Han}

\thanks{Research supported in part by NSERC Canada under Grant
RGP 228051. \hfill  \today
}

\address{Department of Mathematical and Statistical Sciences,
University of Alberta, Edmonton,\quad Alberta, Canada T6G 2G1.
\quad {\tt bhan@math.ualberta.ca}\quad
{\tt http://www.ualberta.ca/$\sim$bhan}
}

\makeatletter \@addtoreset{equation}{section} \makeatother
\begin{abstract}
Continuing the lines developed in \cite{Han:wavelet}, in this paper we study nonhomogeneous wavelet systems in high dimensions. It is of interest to study a wavelet system with a minimum number of generators. It has been showed 
by X. Dai, D. R. Larson, and D. M. Speegle in \cite{DLS}
that for any $d\times d$ real-valued expansive matrix $\dm$, a homogeneous orthonormal $\dm$-wavelet basis can be generated by a single wavelet function.
On the other hand, it has been demonstrated in
\cite{Han:wavelet} that
nonhomogeneous wavelet systems, though much less studied in the literature, play a fundamental role in wavelet analysis and naturally link many aspects of wavelet analysis together.
In this paper, we are interested in nonhomogeneous wavelet systems in high dimensions with a minimum number of generators. As we shall see in this paper, a nonhomogeneous wavelet system naturally leads to a homogeneous wavelet system with almost all properties preserved. We also show that a nonredundant nonhomogeneous wavelet system is naturally connected to refinable structures and has a fixed number of wavelet generators. Consequently, it is often impossible for a nonhomogeneous orthonormal wavelet basis to have a single wavelet generator. However, for redundant nonhomogeneous wavelet systems, we show that for any $d\times d$ real-valued expansive matrix $\dm$, we can always construct a nonhomogeneous smooth tight $\dm$-wavelet frame in $\dLp{2}$ with a single wavelet generator whose Fourier transform is a compactly supported $C^\infty$ function. Moreover, such nonhomogeneous tight wavelet frames are associated with filter banks and can be modified to achieve directionality in high dimensions.
Our analysis of nonhomogeneous wavelet systems employs a notion of frequency-based nonhomogeneous wavelet systems in the distribution space. Such a notion allows us to completely separate the perfect reconstruction property of a wavelet system from its stability in various function spaces.
\end{abstract}


 \maketitle

\bigskip

\pagenumbering{arabic}


\section{Introduction and Motivations}
For a function $f: \dR \rightarrow \C$ and a $d\times d$ real-valued invertible matrix $U$, throughout the paper we shall adopt the following notation:
\begin{equation}\label{tdm}
f_{U; \vk, \vn}(x):=|\det U|^{1/2} e^{-\iu \vn \cdot Ux}f(Ux-\vk) \quad \mbox{and}\quad
f_{U; \vk}:=f_{U;\vk,\0},
 \qquad x, \vk, \vn\in \dR,
\end{equation}
where $\iu$ denotes the imaginary unit. Now we recall the definition of a homogeneous wavelet system, which is closely related to the discretization of  a continuous wavelet transform.
Let $\dm$ be a $d\times d$ real-valued invertible matrix and let $\Psi$ be a subset of square integrable functions in $\dLp{2}$. The following \emph{homogeneous $\dm$-wavelet system}
\begin{equation}\label{hws}
\WS(\Psi):=\{ \psi_{\dm^j;\vk} \; : \; j\in \Z, \vk\in \dZ, \psi\in \Psi\}
\end{equation}
has been extensively studied in the function space $\dLp{2}$ in wavelet analysis, often with $\dm$ being an integer expansive matrix. Here we say that $\dm$ is {\it an expansive matrix} if all its eigenvalues have modulus greater than one. The elements in $\Psi$ of \eqref{hws} are called wavelet functions or wavelet generators. It is important to point out here that the elements in a set $S$ of generators in this paper are not necessarily distinct and $S$ may be an infinite set. The notation $h\in S$ in a summation means that $h$ visits every element (with multiplicity) in $S$ once and only once. For a set $S$, we shall use $\#S$ to denote its cardinality counting multiplicity. For example, for $\Psi=\{\psi^1, \ldots, \psi^\mpsi\}$, its cardinality $\# \Psi$ is $s$ and $\mpsi$ could be $\infty$, all the functions $\psi^1, \ldots, \psi^\mpsi$ are not necessarily distinct, and $\psi\in\Psi$ in \eqref{hws} simply means $\psi=\psi^1, \ldots, \psi^\mpsi$. For some references on homogeneous wavelet systems, see \cite{BL,BS}, \cite{CHS}--\cite{RonShen:twf}.

It is of interest to study a homogeneous wavelet system with a minimum number of generators. It has been showed in the interesting paper of Dai, Larson, and Speegle \cite[Corollary~1]{DLS} that for any real-valued expansive matrix $\dm$, there exists a (Lebesgue) measurable subset $E$ of $\dR$ such that the homogeneous $\dm$-wavelet system $\WS(\{\psi\})$ is an orthonormal basis in $\dLp{2}$, where $\hat \psi=\chi_E$, the characteristic function of $E$. Such a measurable set $E$ is called a dilation-$\dm$ wavelet set in \cite{DL,DLS}. For more discussion on wavelet sets and their constructions, see \cite{BS,DL,DLS,HW} and references therein. Since an orthonormal wavelet basis is a special case of a tight (or Parseval) wavelet frame, such wavelet sets are automatically framelet sets which yield homogeneous tight wavelet frames in $\dLp{2}$ with a single wavelet generator.
For $\dm=2I_d$ where $I_d$ denotes the $d\times d$ identity matrix, it has been showed in \cite[Theorem~3.8]{Han:frame} that one can always construct a homogeneous smooth tight $\dm$-wavelet frame $\WS(\{\psi\})$ with a single generator such that $\hat \psi$ is a compactly supported function in $C^\infty(\dR)$.
The Fourier transform used in this paper for $f\in \dLp{1}$ is defined to be $\hat f(\xi)=\int_{\dR} f(x) e^{-\iu x\cdot \xi} dx, \xi\in \dR$ and can be naturally extended to square integrable functions and tempered distributions.
Through the study of fast wavelet transforms and their underlying wavelet systems, it has been recently noticed in \cite{Han:wavelet} that a nonhomogeneous wavelet system, though much less studied in the literature, plays a fundamental role in understanding many aspects of wavelet analysis and its applications.
In this paper, we shall continue the lines developed in \cite{Han:wavelet} for dimension one to investigate nonhomogeneous wavelet systems in high dimensions. Let $\Phi$ and $\Psi$ be subsets of $\dLp{2}$. \emph{A nonhomogeneous $\dm$-wavelet system} is defined to be
\begin{equation}\label{nws}
\WS_J(\Phi; \Psi):=\{ \phi_{\dm^J;\vk}\; : \; \vk\in \dZ, \phi\in \Phi \} \cup
\{ \psi_{\dm^j;\vk}\; : \; j\ge J, \vk\in \dZ, \psi\in \Psi\},
\end{equation}
where $J$ is an integer representing the coarsest scale level.
In this paper, we are particularly interested in nonhomogeneous wavelet systems with a minimum number of generators, that is, the smallest possible cardinalities $\#\Phi$ and $\#\Psi$ of generators in \eqref{nws}. In contrast to homogeneous wavelet systems,
we shall see in this paper that there is an intrinsic difference between nonredundant and redundant nonhomogeneous wavelet systems.
To have some rough ideas about our results in this paper, let us present here two results on nonredundant and redundant nonhomogeneous wavelet systems with a minimum number of generators.

For nonhomogeneous orthonormal wavelet bases, we have the following result, which is a special case of Theorem~\ref{thm:bw}.

\begin{theorem}\label{thm:rw:number} Let $\dm$ be a $d\times d$ integer invertible matrix. Let $\Phi=\{\phi^1, \ldots, \phi^\mphi\}$ and $\Psi=\{\psi^1, \ldots, \psi^\mpsi\}$ be subsets of $\dLp{2}$ with $\mphi,\mpsi\in \N\cup\{0\}$. Suppose that the nonhomogeneous $\dm$-wavelet system $\WS_J(\Phi; \Psi)$ is an orthonormal basis of $\dLp{2}$ for some integer $J$. Then $\mpsi=\mphi(|\det \dm|-1)$ and there exist an $\mphi\times \mphi$ matrix $\fa$ and an $\mpsi\times \mphi$ matrix $\fb$ of $2\pi$-periodic measurable functions in $\dTLp{2}$ such that
\begin{equation}\label{refeq:phi:psi}
\hat \phi(\dm^T\xi)=\fa(\xi)\hat \phi(\xi)\quad \mbox{and}\quad
\hat \psi(\dm^T\xi)=\fb(\xi)\hat \phi(\xi),\qquad a.e.\, \xi\in \dR,
\end{equation}
where $\phi:=[\phi^1, \ldots, \phi^\mphi]^T$ and $\psi:=[\psi^1, \ldots, \psi^\mpsi]^T$.
If in addition $\dm$ is an expansive matrix, then the homogeneous $\dm$-wavelet system $\WS(\Psi)$ is an orthonormal basis of $\dLp{2}$.
\end{theorem}

By Theorem~\ref{thm:rw:number}, for a nonhomogeneous orthonormal $\dm$-wavelet basis in $\dLp{2}$, the smallest possible numbers of generators are $\#\Phi=1$ and $\#\Psi=1$, for which the integer matrix $\dm$ must satisfy $|\det \dm|=2$.

In contrast to nonredundant nonhomogeneous wavelet systems in Theorem~\ref{thm:rw:number}, for redundant nonhomogeneous wavelet systems such as tight wavelet frames (and consequently dual wavelet frames), we have the following result, whose proof and construction are given in Section~4.

\begin{theorem}\label{thm:ntwf:special} Let $\dm$ be a $d\times d$ real-valued expansive matrix. Then there exist two real-valued functions $\phi, \psi$ in the Schwarz class such that
\begin{enumerate}
\item[{\rm(i)}] $\WS_J(\{\phi\}; \{\psi\})$ is a nonhomogeneous tight $\dm$-wavelet frame in $\dLp{2}$ for all integers $J$:
\begin{equation}\label{ntwf:special}
\|f\|^2_{\dLp{2}}=\sum_{\vk\in \dZ} |\la f, \phi_{\dm^J; \vk}\ra|^2+\sum_{j=J}^\infty \sum_{\vk\in \dZ} |\la f, \psi_{\dm^j; \vk}\ra|^2\qquad \forall\; f\in \dLp{2};
\end{equation}
\item[{\rm(ii)}] $\hat \phi$ and $\hat \psi$ are compactly supported $C^\infty$ even functions;
\item[{\rm(iii)}] $\hat \psi$ vanishes in a neighborhood of the origin, i.e., $\psi$ has arbitrarily high vanishing moments;
\item[{\rm(iv)}] 
there exist $2\pi\dZ$-periodic measurable functions $\fa_\vk, \fb_\vk, \vk\in \dZ$ in $C^\infty(\dT)$ such that
\begin{equation}\label{refeq}
e^{-\iu \vk \cdot \dm^T\xi} \hat \phi(\dm^T \xi)=\fa_\vk(\xi) \hat \phi(\xi) \quad \mbox{and}\quad
e^{-\iu \vk \cdot \dm^T\xi} \hat \psi(\dm^T \xi)=\fb_\vk(\xi) \hat \phi(\xi), \quad \xi\in \dR, \vk\in \dZ.
\end{equation}
\end{enumerate}
Moreover, $\WS(\{\psi\})$ is a homogeneous tight $\dm$-wavelet frame in $\dLp{2}$ satisfying
\begin{equation}\label{twf:special}
\|f\|^2_{\dLp{2}}=\sum_{j\in \Z} \sum_{\vk\in \dZ} |\la f, \psi_{\dm^j; \vk}\ra|^2\qquad \forall\; f\in \dLp{2}.
\end{equation}
\end{theorem}

The property in \eqref{refeq} is called the refinable structure of $\phi$ and $\psi$ in this paper.
We shall see in Section~4 that the nonhomogeneous tight $\dm$-wavelet frames in Theorem~\ref{thm:ntwf:special} can be easily modified to achieve directionality in high dimensions by using nonstationary tight wavelet frames.

The structure of the paper is as follows. In Section~2, we shall explore the connections between nonhomogeneous and homogeneous wavelet systems in $\dLp{2}$. We shall see in Section~2 that any given nonhomogeneous wavelet system will yield a homogeneous wavelet system and a sequence of nonhomogeneous wavelet systems with almost all properties preserved. For a nonredundant nonhomogeneous wavelet system, we shall see in Section~2 that it has a natural connection to refinable structures. In Section~3, we shall introduce and characterize a pair of frequency-based nonhomogeneous (and more generally, nonstationary) dual wavelet frames in the distribution space. We shall see in Section~3 that this notion
allows us to completely separate the perfect reconstruction property of a wavelet system from its stability in various function spaces.
As an application, for any $d\times d$ real-valued invertible matrix $\dm$, we obtain a complete characterization of a pair of nonhomogeneous (as well as nonstationary) dual wavelet frames in $\dLp{2}$. This also yields a characterization of nonstationary tight wavelet frames in $\dLp{2}$.
Based on the results in Section~3, we shall prove Theorem~\ref{thm:ntwf:special} in Section~4.
Moreover, we shall show that such nonhomogeneous tight wavelet frames in Theorem~\ref{thm:ntwf:special} are associated with filter banks and can be easily modified to achieve directionality in high dimensions.

\section{Nonhomogeneous and Homogeneous Wavelet Systems in $\dLp{2}$}

In this section, we shall study the connections of a nonhomogeneous wavelet system in $\dLp{2}$ to a homogeneous wavelet system in $\dLp{2}$. To do so, we need the following auxiliary result.

\begin{lemma}\label{lem:to0}
Let $\dm$ be a $d\times d$ real-valued expansive matrix and $\Phi$ be a (not necessarily finite) subset of $\dLp{2}$ such that $\sum_{\phi\in \Phi} \|\phi\|_{\dLp{2}}^2<\infty$. Suppose that there exists a positive constant $C$ such that $\sum_{\phi\in \Phi} \sum_{\vk\in \dZ} |\la f, \phi(\cdot-\vk)\ra|^2\le C \|f\|^2_{\dLp{2}}$ for all $f\in \dLp{2}$. Then
\begin{equation}\label{to0}
\lim_{j\to -\infty} \sum_{\phi\in \Phi} \sum_{\vk\in \dZ} |\la f, \phi_{\dm^j; \vk}\ra|^2=0\qquad \forall\; f\in \dLp{2}.
\end{equation}
\end{lemma}

\begin{proof} We use a similar argument as in \cite[Lemma~3]{Han:wavelet}. We first prove that \eqref{to0} holds with $f=\chi_E$, the characteristic function of a bounded measurable set $E$. By calculation, for $f=\chi_E$, we have
\[
|\la f, \phi_{\dm^j; \vk}\ra|^2
=|\det \dm|^{j} \Big| \int_E \phi(\dm^j x-\vk)dx\Big|^2
\le |\det \dm|^{-j} \Big(\int_{\dm^{j} E-\vk} |\phi(x)| dx\Big)^2
\le |E| \int_{\dm^j E-\vk} |\phi(x)|^2 dx,
\]
where $|E|$ denotes the Lebesgue measure of $E$. Since $E$ is bounded and $\dm$ is an expansive matrix, we see that $\dm^j E-\vk, \vk\in \dZ$ are mutually disjoint as $j\to -\infty$. Therefore, we have
\[
\sum_{\phi\in \Phi} \sum_{\vk\in \dZ} |\la f, \phi_{\dm^j; \vk}\ra|^2
\le |E| \sum_{\phi\in \Phi} \sum_{\vk\in \dZ} \int_{\dm^j E-\vk} |\phi(x)|^2 dx=
|E| \int_{\cup_{\vk\in \dZ} (\dm^j E-\vk)} h(x) dx,
\]
where $h(x):=\sum_{\phi\in \Phi} |\phi(x)|^2$. By our assumption on $\Phi$, we have $h\in \dLp{1}$. Since $h\in \dLp{1}$, for arbitrary $\gep>0$, there exist positive constants $N$ and $c$ such that
$\int_{\{ y\in \dR : |y|\ge N\}} h(x) dx<\gep$ and $\int_K h(x) dx<\gep$ for every measurable set $K$ with $|K|<c$. Since $\dm$ is an expansive matrix, there exists an integer $J$ such that for all $j\le J$, the Lebesgue measure of the set $K_N:=\{ y\in \dR : |y|\le N\} \cap (\cup_{\vk\in \dZ} (\dm^j E-\vk))$ is less than $c$. Consequently, we deduce that
\[
\sum_{\phi\in \Phi} \sum_{\vk\in \dZ} |\la f, \phi_{\dm^j; \vk}\ra|^2
\le |E| \int_{\cup_{\vk\in \dZ} (\dm^j E-\vk)} h(x) dx
\le |E|\int_{K_N} h(x) dx+|E|\int_{\{ y\in \dR : |y|\ge N\} } h(x) dx
\le 2 |E|\gep
\]
for all $j\le J$. Therefore, we see that \eqref{to0} holds for $f=\chi_E$. Consequently, \eqref{to0} holds for any function $f$ which is a finite linear combination of characteristic functions of some bounded measurable sets.

Define operators $P_j: \dLp{2}\rightarrow l_2(\Z \times \Phi)$ by $P_j f:=\{ \la f, \phi_{\dm^j;\vk}\ra\}_{\vk\in \dZ, \phi\in \Phi}$. Since $\la f, \phi_{\dm^j; \vk}\ra=\la f_{\dm^{-j}; \0}, \phi(\cdot-\vk)\ra$, we can easily deduce that
\[
\|P_j f\|^2:=\sum_{\phi\in \Phi} \sum_{\vk\in \dZ} |\la f, \phi_{\dm^j;\vk}\ra|^2
=\sum_{\phi\in \Phi} \sum_{\vk\in \dZ} |\la f_{\dm^{-j};\0}, \phi(\cdot-\vk)\ra|^2\le C \| f_{\dm^{-j};\0}\|_{\dLp{2}}^2=C \|f\|^2_{\dLp{2}}
\]
for all $j\in \Z$. For any $f\in \dLp{2}$ and arbitrary $\gep>0$, there exists a function $g$, which is a finite linear combination of characteristic functions of bounded measurable sets, such that
$\|f-g\|_{\dLp{2}}\le \gep$. Since we proved that $\lim_{j\to -\infty} \|P_j g\|=0$, there exists $J\in \Z$ such that $\|P_j g\| \le \gep$ for all $j\le J$. Consequently, we have
\[
\|P_j f\| \le \|P_j g\|+\|P_j (f-g)\| \le \gep+\sqrt{C} \|f-g\|_{\dLp{2}} \le \gep(1+\sqrt{C}), \qquad \forall\; j\le J.
\]
Therefore, we conclude that $\lim_{j\to -\infty} \|P_j f\|=0$. That is, \eqref{to0} holds.
\end{proof}

We mention that \eqref{to0} in Lemma~\ref{lem:to0} holds for a more general $d\times d$ real-valued (not necessarily expansive) matrix. For example, by a slightly modified proof of Lemma~\ref{lem:to0}, \eqref{to0} holds if $\dm$ is a $d\times d$ real-valued matrix such that
\begin{equation}\label{dm:special}
|\det \dm|>1 \qquad \mbox{and}\qquad \limsup_{j\to -\infty} \sup_{\|x\|\le 1} \|\dm^j x\|<\infty.
\end{equation}
Without requiring that $\dm$ be an expansive matrix, all the results in this section still hold provided that \eqref{to0} holds even with $\lim$ in \eqref{to0} being replaced by $\liminf$.

\begin{prop}\label{prop:frame}
Let $\dm$ be a $d\times d$ real-valued invertible matrix. Let $\Phi$ and $\Psi$ be subsets of $\dLp{2}$. Suppose that $\WS_J(\Phi; \Psi)$ is a nonhomogeneous $\dm$-wavelet frame in $\dLp{2}$ for some integer $J$, that is, there exist positive constants $C_1$ and $C_2$ such that
\begin{equation}\label{nws:frame}
\begin{split}
C_1 \|f\|_{\dLp{2}}^2 &\le \sum_{\phi\in \Phi} \sum_{\vk\in \dZ} |\la f, \phi_{\dm^J; \vk}\ra|^2+
\sum_{j=J}^\infty \sum_{\psi\in \Psi} \sum_{\vk\in \dZ} |\la f, \psi_{\dm^j; \vk}\ra|^2
\le C_2 \|f\|_{\dLp{2}}^2,\\
&\qquad\qquad\qquad\qquad \qquad \forall \, f\in \dLp{2}.
\end{split}
\end{equation}
Then \eqref{nws:frame} holds for all integers $J$, in other words, $\WS_J(\Phi; \Psi)$ is a nonhomogeneous $\dm$-wavelet frame with the same lower and upper frame bounds for all integers $J$. If in addition $\dm$ is an expansive matrix and $\sum_{\phi\in \Phi} \|\phi\|_{\dLp{2}}^2<\infty$, then the homogeneous $\dm$-wavelet system $\WS(\Psi)$ is also a frame in $\dLp{2}$ with the same lower and upper frame bounds:
\begin{equation}\label{hws:frame}
C_1 \|f\|_{\dLp{2}}^2\le \sum_{j\in \Z} \sum_{\psi\in \Psi}
\sum_{\vk\in \dZ} |\la f, \psi_{\dm^j; \vk}\ra|^2\le C_2 \|f\|_{\dLp{2}}^2, \qquad \forall \, f\in \dLp{2}.
\end{equation}
\end{prop}

\begin{proof} By the following simple fact that for $f,g\in \dLp{2}$ and invertible $d\times d$ matrices $U$ and $V$,
\begin{equation}\label{simple}
\la f_{U; \vk,\vn}, g_{U; \vk,\vn} \ra =\la f, g\ra \quad\mbox{and} \quad \la f_{V; \0,\0}, \psi_{U; \vk,\vn}\ra=\la f, \psi_{U V^{-1}; \vk,\vn}\ra,\qquad \vk,\vn\in \dZ,
\end{equation}
it is straightforward to see that if \eqref{nws:frame} holds for one integer $J$, then \eqref{nws:frame} holds for all integers $J$.

Next we prove that \eqref{hws:frame} holds. Since \eqref{nws:frame} holds for all integers $J$, it is easy to see that
\[
\sum_{j=J}^\infty \sum_{\psi\in \Psi} \sum_{\vk\in \dZ} |\la f, \psi_{\dm^j; \vk}\ra|^2\le C_2 \|f\|_{\dLp{2}}^2
\]
and
\[
C_1\|f\|^2_{\dLp{2}}-\sum_{\phi\in \Phi} \sum_{\vk\in \dZ} |\la f, \phi_{\dm^J; \vk}\ra|^2\le
\sum_{j\in\Z} \sum_{\psi\in \Psi} \sum_{\vk\in \dZ} |\la f, \psi_{\dm^j; \vk}\ra|^2
\]
for all integers $J\in \Z$ and $f\in \dLp{2}$.
Taking $J\to -\infty$ in the above two inequalities and using \eqref{to0} for the second inequality, we deduce that \eqref{hws:frame} holds.
\end{proof}

The best possible constants $C_1$ and $C_2$ in \eqref{nws:frame} are called the lower frame bound and the upper frame bound of $\WS_J(\Phi; \Psi)$, respectively.

Let $\dm$ be a $d\times d$ real-valued invertible matrix. Let
\begin{equation}\label{Phi:Psi}
\Phi=\{\phi^1, \ldots, \phi^\mphi\}, \quad
\Psi=\{\psi^1, \ldots, \psi^\mpsi\} \quad\mbox{and}\quad
\tilde \Phi=\{\tilde\phi^1, \ldots, \tilde \phi^\mphi\},\quad
 \tilde \Psi=\{\tilde \psi^1, \ldots, \tilde \psi^\mpsi\}
\end{equation}
be subsets of $\dLp{2}$, where $\mphi, \mpsi\in \N \cup \{0, +\infty\}$.
Let $\WS_J(\Phi; \Psi)$ be defined in \eqref{nws} and $\WS_J(\tilde \Phi; \tilde \Psi)$ be defined similarly.
We say that the pair
$(\WS_J(\Phi; \Psi), \WS_J(\tilde \Phi; \tilde \Psi))$
is {\it a pair of nonhomogeneous dual $\dm$-wavelet frames in $\dLp{2}$} if each of $\WS_J(\Phi; \Psi)$ and $\WS_J(\tilde \Phi; \tilde \Psi)$ is a nonhomogeneous $\dm$-wavelet frame in $\dLp{2}$ and the following identity holds
\begin{equation}\label{dwf}
\la f, g\ra=\sum_{\ell=1}^\mphi \sum_{\vk\in \dZ} \la f, \phi^\ell_{\dm^J; \vk}\ra\la \tilde \phi^\ell_{\dm^J; \vk}, g\ra+\sum_{j=J}^\infty \sum_{\ell=1}^\mpsi \sum_{\vk\in \dZ} \la f, \psi^\ell_{\dm^j; \vk}\ra\la \tilde \psi^\ell_{\dm^j; \vk}, g\ra, \quad \forall\; f,g \in \dLp{2}.
\end{equation}

For pairs of nonhomogeneous dual wavelet frames in $\dLp{2}$, we have the following result.

\begin{prop}\label{prop:dwf}
Let $\dm$ be a $d\times d$ real-valued invertible matrix.
Let $\Phi,\Psi, \tilde \Phi, \tilde \Psi$ in \eqref{Phi:Psi} be subsets of $\dLp{2}$. Suppose that $(\WS_J(\Phi; \Psi), \WS_J(\tilde \Phi; \tilde \Psi))$
is a pair of nonhomogeneous dual $\dm$-wavelet frames in $\dLp{2}$ for some integer $J$. Then it is a pair of nonhomogeneous dual $\dm$-wavelet frames in $\dLp{2}$ for all integers $J$. If in addition $\dm$ is an expansive matrix and $\sum_{\phi\in \Phi} \|\phi\|_{\dLp{2}}^2<\infty$, then the pair $(\WS(\Psi),\WS(\tilde \Psi))$
is a pair of homogeneous dual $\dm$-wavelet frames in $\dLp{2}$, that is, each of $\WS(\Psi)$ and $\WS(\tilde \Psi)$ is a frame in $\dLp{2}$ and the following identity holds:
\begin{equation}\label{hws:dwf}
\la f, g\ra= \sum_{j\in \Z} \sum_{\ell=1}^\mpsi\sum_{\vk\in \dZ} \la f, \psi^\ell_{\dm^j; \vk}\ra\la \tilde \psi^\ell_{\dm^j; \vk}, g\ra, \qquad \forall\; f,g \in \dLp{2}.
\end{equation}
\end{prop}

\begin{proof} By \eqref{simple}, we see that \eqref{dwf} holds for all integers $J$. By Proposition~\ref{prop:frame}, $(\WS_J(\Phi; \Psi), \WS_J(\tilde \Phi; \tilde \Psi))$ is a pair of nonhomogeneous dual $\dm$-wavelet frames in $\dLp{2}$ for all integers $J$.
For a real-valued expansive matrix $\dm$, by Proposition~\ref{prop:frame} again, each of $\WS(\Psi)$ and $\WS(\tilde \Psi)$ is a frame in $\dLp{2}$. By Lemma~\ref{lem:to0}, the identity in \eqref{hws:dwf} can be proved similarly as in the proof of Proposition~\ref{prop:frame}.
\end{proof}

Next, we discuss nonredundant nonhomogeneous wavelet systems.
Let $\Phi$ and $\Psi$ be subsets of $\dLp{2}$. We say that $\WS_J(\Phi; \Psi)$ is {\it a nonhomogeneous Riesz $\dm$-wavelet basis} in $\dLp{2}$ if the linear span of
$\WS_J(\Phi; \Psi)$
is dense in $\dLp{2}$ and there exist two positive constants $C_3$ and $C_4$ such that
{\small \begin{equation}\label{riesz}
\begin{split}
C_3 \Big( \sum_{\phi\in \Phi} \sum_{\vk\in \dZ} |v_{\vk,\phi}|^2
\!+\!\sum_{j=J}^\infty \sum_{\psi\in \Psi} \sum_{\vk\in \dZ} |w_{j;\vk,\psi}|^2\Big) &\!\le\!
\Big\| \sum_{\phi\in \Phi} \sum_{\vk\in \dZ} v_{\vk,\phi} \phi_{\dm^J; \vk}
\!+\!\sum_{j=J}^\infty \sum_{\psi\in \Psi} \sum_{\vk\in \dZ} w_{j;\vk,\psi} \psi_{\dm^j; \vk}\Big\|^2_{\dLp{2}}\\
&\le C_4 \Big( \sum_{\phi\in \Phi} \sum_{\vk\in \dZ} |v_{\vk,\phi}|^2
+\sum_{j=J}^\infty \sum_{\psi\in \Psi} \sum_{\vk\in \dZ} |w_{j;\vk,\psi}|^2\Big)
\end{split}
\end{equation}
}
for all finitely supported sequences $\{ v_{\vk,\phi}\}_{\vk\in \dZ,\phi\in \Phi}$ and $\{ w_{j;\vk,\psi}\}_{j\ge J, \vk\in \dZ,\psi\in \Psi}$, where the best possible constants $C_3$ and $C_4$ are called the lower Riesz bound and the upper Riesz bound of $\WS_J(\Phi;\Psi)$, respectively. It is well known that
$\WS_J(\Phi;\Psi)$ is a Riesz basis of $\dLp{2}$ if and only if it is a frame of $\dLp{2}$ and it is {\it $l_2$-linearly independent}, that is, if
\begin{equation}\label{l2:li}
\sum_{\phi\in \Phi} \sum_{\vk\in \dZ} v_{\vk,\phi} \phi_{\dm^J; \vk}
+\sum_{j=J}^\infty \sum_{\psi\in \Psi} \sum_{\vk\in \dZ} w_{j;\vk,\psi} \psi_{\dm^j; \vk}=0
\end{equation}
in $\dLp{2}$ for square summable sequences $\{ v_{\vk,\phi}\}_{\vk\in \dZ, \phi\in \Phi}$ and $\{ w_{j;\vk,\psi}\}_{j\ge J, \vk\in \dZ,\psi\in \Psi}$, then $v_{\vk,\phi}=0$ and $w_{j;\vk,\psi}=0$ for all $j\ge J$, $\vk\in \dZ$ and $\phi\in \Phi, \psi\in \Psi$. In fact, define $\Wop: l_2(\dZ\times \Phi)\times l_2(\Z\times \dZ\times \Psi)\rightarrow \dLp{2}$ such that $\Wop$ maps $(\{ v_{\vk,\phi}\}_{\vk\in \dZ, \phi\in \Phi}, \{ w_{j;\vk,\psi}\}_{j\ge J, \vk\in \dZ,\psi\in \Psi})$ to the function on the left-hand side of \eqref{l2:li}. Then $\WS_J(\Phi; \Psi)$ is a Riesz basis in $\dLp{2}$ if and only if $\Wop$ is a well-defined bounded and invertible operator. Also, $\WS_J(\Phi; \Psi)$ is a frame in $\dLp{2}$ if and only if $\Wop$ is a well-defined bounded onto operator with the range of its adjoint operator $\Wop^*$ being closed.
Now it is straightforward to see that $\WS_J(\Phi;\Psi)$ is a Riesz basis of $\dLp{2}$ if and only if it is a frame of $\dLp{2}$ and it is $l_2$-linearly independent. Thus, a nonhomogeneous Riesz $\dm$-wavelet basis is nonredundant. Note that a nonhomogeneous orthonormal $\dm$-wavelet basis is a particular case of a nonhomogeneous Riesz $\dm$-wavelet basis.

For nonhomogeneous Riesz wavelet bases in $\dLp{2}$, we have the following result.

\begin{theorem}\label{thm:rw} Let $\dm$ be a $d\times d$ real-valued invertible matrix. Let $\Phi$ and $\Psi$ be subsets of $\dLp{2}$. Suppose that $\WS_J(\Phi; \Psi)$ is a nonhomogeneous Riesz $\dm$-wavelet basis in $\dLp{2}$ satisfying \eqref{riesz} for some integer $J$. Then \eqref{riesz} holds for all integers $J$ and $\WS_J(\Phi; \Psi)$ is a nonhomogeneous Riesz $\dm$-wavelet basis in $\dLp{2}$ with the same lower and upper Riesz bounds  for all integers $J$.
If $\dm$ is a $d\times d$ real-valued expansive matrix and $\sum_{\phi\in \Phi} \|\phi\|^2_{\dLp{2}}<\infty$, then $\WS(\Psi)$ is a homogeneous Riesz $\dm$-wavelet basis in $\dLp{2}$ with the same lower and upper Riesz bounds, that is, the linear span of $\WS(\Psi)$ is dense in $\dLp{2}$ and
\begin{equation}\label{h:riesz}
C_3 \sum_{j\in \Z} \sum_{\psi\in \Psi} \sum_{\vk\in \dZ} |w_{j;\vk,\psi}|^2 \le
\Big\| \sum_{j\in \Z} \sum_{\psi\in \Psi} \sum_{\vk\in \dZ} w_{j;\vk,\psi} \psi_{\dm^j; \vk}\Big\|^2_{\dLp{2}}
\le C_4 \sum_{j\in \Z} \sum_{\psi\in \Psi} \sum_{\vk\in \dZ} |w_{j;\vk,\psi}|^2
\end{equation}
for all finitely supported sequences $\{w_{j;\vk,\psi}\}_{j\in \Z, \vk\in \dZ, \psi\in \Psi}$.
\end{theorem}

\begin{proof}
Since \eqref{riesz} holds, by $(\psi_{\dm^j;\vk})_{\dm^n;\0}=\psi_{\dm^{j+n};\vk}$ and $\|f_{\dm^j;\0}\|_{\dLp{2}}^2=\|f\|_{\dLp{2}}^2$, it is easy to verify that \eqref{riesz} holds for all integers $J$.
Note that $\WS_{J+n}(\Phi; \Psi)=\{ f_{\dm^n;\0}\; : \; f\in \WS_J(\Phi; \Psi)\}$ for all integers $J$ and $n$.
It is also easy to deduce that if the linear span of $\WS_J(\Phi; \Psi)$ is dense in $\dLp{2}$ for one integer $J$, then it is also dense in $\dLp{2}$ for all integers $J$.  Consequently, $\WS_J(\Phi; \Psi)$ is a nonhomogeneous Riesz $\dm$-wavelet basis in $\dLp{2}$ for all integers $J$ with the same lower and upper Riesz bounds.

Since \eqref{riesz} holds for all integers $J$, setting $v_{\vk,\phi}=0$ in \eqref{riesz}, we can easily deduce that \eqref{h:riesz} holds  for all finitely supported sequences $\{w_{j;\vk,\psi}\}_{j\in \Z, \vk\in \dZ, \psi\in \Psi}$. Since $\WS_J(\Phi; \Psi)$ is a Riesz basis in $\dLp{2}$,
we see that $\WS_J(\Phi;\Psi)$ must be a frame in $\dLp{2}$. Since $\dm$ is an expansive matrix and $\sum_{\phi\in \Phi} \|\phi\|_{\dLp{2}}^2<\infty$, by Proposition~\ref{prop:frame}, $\WS(\Psi)$ is a frame in $\dLp{2}$.
Therefore, by \eqref{hws:frame}, we see that $f\perp \WS(\Psi)$ implies $f=0$. Thus, we conclude that the linear span of $\WS(\Psi)$ is dense in $\dLp{2}$. Therefore, $\WS(\Psi)$ is a Riesz basis of $\dLp{2}$.
\end{proof}

Let $\gd$ denote the {\it Dirac sequence} such that $\gd(0)=1$ and $\gd(\vk)=0$ for all $\vk\ne 0$.  Let $\Phi, \Psi, \tilde \Phi, \tilde \Psi$ in \eqref{Phi:Psi} be subsets of $\dLp{2}$. We say that the pair
$(\WS_J(\Phi;\Psi), \WS_J(\tilde \Phi; \tilde \Psi))$
is {\it a pair of nonhomogeneous biorthogonal $\dm$-wavelet bases} in $\dLp{2}$ if each of $\WS_J(\Phi; \Psi)$ and $\WS_J(\tilde \Phi; \tilde \Psi)$ is a nonhomogeneous Riesz $\dm$-wavelet basis in $\dLp{2}$ and the following biorthogonality relation holds:
\begin{align}
&\la \phi^\ell_{\dm^J;\vk}, \tilde \phi^{\ell'}_{\dm^J; \vk'}\ra=\gd(\vk-\vk')\gd(\ell-\ell'),\quad  \la \psi^n_{\dm^j;\vk}, \tilde \phi^{\ell'}_{\dm^{J}; \vk'}\ra=0,\quad \la \phi^\ell_{\dm^J;\vk}, \tilde \psi^{n'}_{\dm^{j'}; \vk'}\ra=0,\label{bio:phi}\\
&\la \psi^n_{\dm^j;\vk}, \tilde \psi^{n'}_{\dm^{j'}; \vk'}\ra=\gd(\vk-\vk')\gd(n-n')\gd(j-j'),\label{bio:psi}
\end{align}
for all $\vk, \vk'\in \dZ$, $j,j'\in \Z$ such that $j, j'\ge J$, $\ell, \ell'=1, \ldots, \mphi$, and $n, n'=1, \ldots, \mpsi$.

It is a standard result, which can be easily proved by the same argument after \eqref{l2:li}, that a pair $(\WS_J(\Phi;\Psi), \WS_J(\tilde \Phi; \tilde \Psi))$ is a pair of nonhomogeneous biorthogonal $\dm$-wavelet bases in $\dLp{2}$ if and only if it is a pair of nonhomogeneous dual $\dm$-wavelet frames in $\dLp{2}$ and the biorthogonality conditions in \eqref{bio:phi} and \eqref{bio:psi} are satisfied.
For pairs of nonhomogeneous biorthogonal $\dm$-wavelet bases in $\dLp{2}$, we have the following result which includes Theorem~\ref{thm:rw:number} as a special case:

\begin{theorem}\label{thm:bw}
Let $\dm$ be a $d\times d$ real-valued invertible matrix.
Let $\Phi,\Psi, \tilde \Phi, \tilde \Psi$ in \eqref{Phi:Psi} be finite subsets of $\dLp{2}$. Suppose that the pair $(\WS_J(\Phi;\Psi), \WS_J(\tilde \Phi; \tilde \Psi))$ is a pair of nonhomogeneous biorthogonal $\dm$-wavelet bases in $\dLp{2}$ for some integer $J$. Then
it is a pair of nonhomogeneous biorthogonal $\dm$-wavelet bases in $\dLp{2}$ for all integers $J$.
Moreover, there exist $\mphi\times \mphi$ matrices $\fa_\vk, \tilde \fa_\vk$ and $\mpsi\times \mphi$ matrices $\fb_\vk, \tilde \fb_\vk, \vk\in \dZ$ of $2\pi\dZ$-periodic measurable functions in $\dTLp{2}$ such that
\begin{align}
&e^{-\iu \vk\cdot \dm^T\xi} \hat\phi(\dm^T \xi)=\fa_\vk(\xi)\hat\phi(\xi) \quad \mbox{and}\quad
e^{-\iu \vk\cdot \dm^T\xi} \hat\psi(\dm^T \xi)=\fb_\vk(\xi) \hat\phi(\xi), \quad a.e.\; \xi\in \dR,\; \vk\in\dZ, \label{phi:psi:refinable}\\
&e^{-\iu \vk\cdot \dm^T\xi} \hat{\tilde \phi}(\dm^T \xi)=\tilde \fa_\vk(\xi)\hat{\tilde \phi}(\xi) \quad \mbox{and}\quad
e^{-\iu \vk\cdot \dm^T\xi} \hat{\tilde \psi}(\dm^T\xi)=\tilde \fb_\vk(\xi) \hat{\tilde\phi}(\xi), \quad a.e.\; \xi\in \dR, \; \vk\in \dZ, \label{tphi:tpsi:refinable}
\end{align}
where
\begin{equation}\label{phi:psi}
\phi:=[\phi^1, \ldots, \phi^\mphi]^T, \qquad \psi:=[\psi^1, \ldots, \psi^\mpsi]^T,\qquad
\tilde \phi:=[\tilde \phi^1, \ldots, \tilde \phi^\mphi]^T, \qquad \tilde \psi:=[\tilde \psi^1, \ldots, \tilde \psi^\mpsi]^T.
\end{equation}
If $\dm$ is a $d\times d$ integer invertible matrix with $\ddm:=|\det \dm|$, then $\mpsi=\mphi(\ddm-1)$ and
\begin{equation}\label{fb:bw}
\ol{\PP_{[\tilde\fa_\0, \tilde\fb_\0]}(\xi)}^T \PP_{[\fa_\0, \fb_\0]}(\xi)=I_{\mphi \ddm},\qquad a.e.\, \xi\in \dR,
\end{equation}
where $I_{\mphi\ddm}$ denotes the $(\mphi\ddm)\times(\mphi\ddm)$ identity matrix and
\begin{equation}\label{PP}
\PP_{[\fa_\0, \fb_\0]}(\xi):=
\left[ \begin{matrix} \fa_\0(\xi+2\pi\omega_0) &\fa_\0(\xi+2\pi \omega_1) &\cdots &\fa_\0(\xi+2\pi\omega_{\ddm-1})\\
\fb_\0(\xi+2\pi\omega_0) &\fb_\0(\xi+2\pi \omega_1) &\cdots &\fb_\0(\xi+2\pi\omega_{\ddm-1})\end{matrix}\right]
\end{equation}
and
$\{\omega_0, \ldots, \omega_{\ddm-1}\}:=[(\dm^T)^{-1} \dZ]\cap [0,1)^d$.
If $\dm$ is a $d\times d$ real-valued expansive matrix, then $(\WS(\Psi), \WS(\tilde \Psi))$ is a pair of homogeneous biorthogonal $\dm$-wavelet bases in $\dLp{2}$, that is, each of $\WS(\Psi)$ and $\WS(\tilde \Psi)$ is a homogeneous Riesz $\dm$-wavelet basis in $\dLp{2}$ and \eqref{bio:psi} holds for all $\vk, \vk'\in \dZ$, $j, j'\in \Z$, and $n, n'=1, \ldots, \mpsi$.
\end{theorem}

\begin{proof}  By Theorem~\ref{thm:rw}, each of $\WS_J(\Phi;\Psi)$ and $\WS_J(\tilde \Phi; \tilde \Psi)$ is a nonhomogeneous Riesz $\dm$-wavelet basis in $\dLp{2}$ for all integers $J$.
By \eqref{simple}, we see that \eqref{bio:phi} and \eqref{bio:psi} must hold for all integers $J$. Thus, $(\WS_J(\Phi;\Psi), \WS_J(\tilde \Phi; \tilde \Psi))$ is a pair of nonhomogeneous biorthogonal $\dm$-wavelet bases in $\dLp{2}$ for all integers $J$.

To prove \eqref{phi:psi:refinable} and \eqref{tphi:tpsi:refinable},
let us consider the representations of $\phi^\ell_{\dm^{-1};\vk}, \psi^n_{\dm^{-1}; \vk}, \vk\in \dZ$, $\ell=1, \ldots, \mphi$ and $n=1, \ldots, \mpsi$ under the Riesz basis $\WS_0(\Phi;\Psi)$. Noting that \eqref{bio:phi} and \eqref{bio:psi} hold for all integers $J$, we deduce that
\begin{equation}\label{phi:psi:ref}
\phi^\ell_{\dm^{-1};\vk}=\sum_{\ell'=1}^\mphi \sum_{\vm\in \dZ} \la \phi^\ell_{\dm^{-1};\vk}, \tilde \phi^{\ell'}_{I_d;\vm}\ra \phi^{\ell'}_{I_d;\vm}
\quad\mbox{and}\quad
\psi^n_{\dm^{-1};\vk}=\sum_{\ell'=1}^\mphi \sum_{\vm\in \dZ} \la \psi^n_{\dm^{-1}; \vk}, \tilde \phi^{\ell'}_{I_d;\vm}\ra \phi^{\ell'}_{I_d;\vm}
\end{equation}
with all the coefficient sequences being square summable and the series in \eqref{phi:psi:ref} converging in $\dLp{2}$.
For $\vk\in \dZ$, define
\begin{equation}\label{fa:fb}
\begin{split}
&[\fa_\vk(\xi)]_{\ell, \ell'}:=|\det \dm|^{-1/2} \sum_{\vm\in \dZ} \la \phi^\ell_{\dm^{-1};\vk}, \tilde \phi^{\ell'}_{I_d; \vm}\ra e^{-\iu \vm \cdot \xi},\\
&[\fb_\vk(\xi)]_{n, \ell'}:=|\det \dm|^{-1/2} \sum_{\vm\in \dZ} \la \psi^n_{\dm^{-1};\vk}, \tilde \phi^{\ell'}_{I_d; \vm}\ra e^{-\iu \vm \cdot \xi},
\end{split}
\end{equation}
where $[\fa_\vk]_{\ell,\ell'}$ denotes the $(\ell,\ell')$-entry of the matrix $\fa_\vk$. It is evident that both
$\fa_\vk$ and $\fb_\vk$ are matrices of $2\pi\dZ$-periodic measurable functions in $\dTLp{2}$.
Taking Fourier transform on both sides of the equations in \eqref{phi:psi:ref} and noting $\wh{\phi^\ell_{\dm^{-1};\vk}}(\xi)= |\det\dm|^{1/2} e^{-i \vk\cdot \dm^T\xi} \hat \phi(\dm^T\xi)$,
we conclude that \eqref{phi:psi:refinable} holds.  Using the same argument by switching the roles of $\phi, \psi$ with $\tilde \phi, \tilde \psi$,
we see that \eqref{tphi:tpsi:refinable} holds with
\begin{equation}\label{tfa:tfb}
\begin{split}
&[\tilde\fa_\vk(\xi)]_{\ell, \ell'}:=|\det \dm|^{-1/2} \sum_{\vm\in \dZ} \la \tilde \phi^\ell_{\dm^{-1};\vk}, \phi^{\ell'}_{I_d; \vm}\ra e^{-\iu \vm \cdot \xi},\\
&[\tilde \fb_\vk(\xi)]_{n, \ell'}:=|\det \dm|^{-1/2} \sum_{\vm\in \dZ} \la \tilde \psi^n_{\dm^{-1};\vk}, \phi^{\ell'}_{I_d; \vm}\ra e^{-\iu \vm \cdot \xi}.
\end{split}
\end{equation}
%

For an integer invertible matrix $\dm$, we denote by
$\{\gamma_0, \ldots, \gamma_{\ddm-1}\}:=\{ \dm x \, : \, x\in (\dm^{-1}\dZ)\cap [0,1)^d\}$. That is, $\{\gamma_0, \ldots, \gamma_{\ddm-1}\}$ is a complete set of representatives of the distinct cosets in the quotient group $\dZ/(\dm\dZ)$.
For $\fa_\0(\xi)=\sum_{\vk\in \dZ} a_\0(k) e^{-\iu \vk\cdot \xi}$, we have $\fa_\0(\xi)=\sum_{n=0}^{\ddm-1} \fa_{\0,\gamma_n}(\dm^T\xi)e^{-\iu \gamma_n \cdot \xi}$ with $\fa_{\0,\gamma_n}(\xi):=\sum_{\vk\in \dZ} a_\0(\dm\vk+\gamma_n) e^{-\iu \vk\cdot \xi}$. Now it is easy to check that
\begin{equation}\label{PP:otherform}
[\fa_\0(\xi+2\pi \omega_0), \ldots, \fa_\0(\xi+2\pi \omega_{\ddm-1})]
=[\fa_{\0,\gamma_0}(\dm^T\xi), \ldots,
\fa_{\0,\gamma_{\ddm-1}}(\dm^T\xi)] \fU(\xi),
\end{equation}
where
\[
\fU(\xi):=(e^{-\iu (\xi+2\pi \omega_m)\cdot \gamma_n} I_\mphi)_{0\le m, n\le \ddm-1} \qquad \mbox{and}\qquad \fU(\xi) \ol{\fU(\xi)}^T=\ddm I_{\mphi \ddm}.
\]
Denote
\[
h:=[\phi^1, \ldots, \phi^\mphi, \psi^1, \ldots, \psi^\mpsi]^T
\quad\mbox{and}\quad
\eta:=[\phi^1_{\dm; \gamma_0}, \ldots, \phi^\mphi_{\dm, \gamma_0},
\ldots, \phi^1_{\dm; \gamma_{\ddm-1}}, \ldots, \phi^\mphi_{\dm, \gamma_{\ddm-1}}]^T.
\]
Noting that $\wh{\phi^\ell_{\dm; \gamma_n}}(\xi)=\ddm^{-1/2} e^{-\iu \gamma_n \cdot \dn\xi} \wh{\phi^\ell}(\dn \xi)$ with $\dn:=(\dm^T)^{-1}$,
we deduce from \eqref{phi:psi:ref} and \eqref{PP:otherform} that
\begin{equation}\label{eta1:eta2}
\hat{h}(\xi)=\ddm^{1/2} \PP_{[\fa_\0, \fb_\0]}(\dn\xi) \fU^{-1}(\dn \xi) \hat{\eta}(\xi).
\end{equation}
Representing the entries in the vector $\eta$ under the Riesz basis $\WS_0(\Phi;\Psi)$, we have
\begin{equation}\label{phi:1:rep}
\phi^\ell_{\dm; \gamma_n}=\sum_{\ell'=1}^\mphi \sum_{\vk\in \dZ}
\la \phi^\ell_{\dm; \gamma_n}, \tilde \phi^{\ell'}_{I_d; -\vk}\ra
\phi^{\ell'}_{I_d; -\vk}+\sum_{n'=1}^\mpsi \sum_{\vk\in \dZ}
\la \phi^\ell_{\dm; \gamma_n}, \tilde \psi^{n'}_{I_d; -\vk}\ra
\psi^{n'}_{I_d; -\vk}.
\end{equation}
Using \eqref{tfa:tfb} and \eqref{phi:1:rep}, by $\la \phi^\ell_{\dm; \gamma_n}, \tilde \phi^{\ell'}_{I_d; -\vk}\ra=\la \phi^\ell_{\dm; \dm\vk+\gamma_n}, \tilde \phi^{\ell'}\ra$, we deduce that
\begin{equation}\label{eta2:eta1}
\hat{\eta}(\xi)=\ddm^{1/2} \ol{\fU^{-1}(\dn \xi)}^T \ol{\PP_{[\tilde \fa_\0, \tilde \fb_\0]}(\dn\xi)}^T \hat{h}(\xi)=\ddm^{-1/2} \fU(\dn \xi) \ol{\PP_{[\tilde \fa_\0, \tilde \fb_\0]}(\dn\xi)}^T \hat{h}(\xi).
\end{equation}
Note that both $h(\cdot-\vk), \vk\in \dZ$
and $\eta(\cdot-\vk), \vk\in \dZ$
are two Riesz bases for the same subspace.
Now it follows from \eqref{eta1:eta2} and \eqref{eta2:eta1} that
\[
\PP_{[\fa_\0, \fb_\0]}(\dn\xi) \ol{\PP_{[\tilde \fa_\0, \tilde \fb_\0]}(\dn\xi)}^T=I_{\mphi+\mpsi} \qquad \mbox{and}\qquad
\ol{\PP_{[\tilde \fa_\0, \tilde \fb_\0]}(\dn\xi)}^T \PP_{[\fa_\0, \fb_\0]}(\dn\xi)=I_{\mphi \ddm}.
\]
The above identities hold if and only if $\mphi+\mpsi=\mphi\ddm$ (that is, $\mpsi=\mphi(\ddm-1))$ and \eqref{fb:bw} holds.

For a real-valued expansive matrix $\dm$, by Theorem~\ref{thm:rw}, each of $\WS(\Psi)$ and  $\WS(\tilde \Psi)$ is a Riesz basis in $\dLp{2}$. Since \eqref{bio:psi} holds for all integers $j$ and $j'$, we conclude that $(\WS(\Psi),\WS(\tilde \Psi))$ is a pair of homogeneous biorthogonal $\dm$-wavelet bases in $\dLp{2}$.
\end{proof}


Using an argument in \cite[Lemma~1]{Han:pams:2006}, we explore in the following result the connections between nonhomogeneous Riesz wavelet bases and refinable structure.

\begin{theorem}\label{thm:rw:refinable} Let $\dm$ be a $d\times d$ integer invertible matrix. Let $\Phi$ and $\Psi$ in \eqref{Phi:Psi} be finite subsets of $\dLp{2}$. Suppose that $\WS_J(\Phi; \Psi)$ is a nonhomogeneous Riesz $\dm$-wavelet basis in $\dLp{2}$ for some integer $J$. Then the following statements are equivalent to each other:
\begin{enumerate}
\item[{\rm(i)}] there exist subsets
$\tilde \Phi$ and $\tilde \Psi$ in \eqref{Phi:Psi} of $\dLp{2}$ such that the pair $(\WS_J(\Phi;\Psi), \WS_J(\tilde \Phi; \tilde \Psi))$ is a pair of nonhomogeneous biorthogonal $\dm$-wavelet bases in $\dLp{2}$;
\item[{\rm(ii)}] there exist an $\mphi\times \mphi$ matrix $\fa$ and an $\mpsi\times \mphi$ matrix $\fb$ of $2\pi\dZ$-periodic measurable functions in $\dTLp{2}$ such that \eqref{refeq:phi:psi} holds with $\phi$ and $\psi$ being defined in \eqref{phi:psi}.
\end{enumerate}
\end{theorem}

\begin{proof} (i)$\Rightarrow$(ii) is a direct consequence of Theorem~\ref{thm:bw}. We now prove (ii)$\Rightarrow$(i).
Since $\WS_J(\Phi; \Psi)$ is a Riesz basis in $\dLp{2}$, it has a unique dual Riesz basis $\{ \phi^{\ell';J,\vk'}\, : \, \vk'\in \dZ, \ell'=1, \ldots, \mphi\} \cup \{\tilde \psi^{n';j',\vk'}\, : \, \vk'\in \dZ, j'\ge J, n'=1, \ldots, \mpsi\}$ such that \eqref{bio:phi} and \eqref{bio:psi} hold with $\tilde \phi^{\ell'}_{\dm^J; \vk'}$ and $\tilde \psi^{n'}_{\dm^{j'}; \vk'}$ being replaced by
$\tilde \phi^{\ell';J,\vk'}$ and $\tilde \psi^{n'; j',\vk'}$, respectively. Define
\begin{equation}\label{drw}
\tilde \phi^{\ell'}:=\tilde \phi^{\ell';J,\0}_{\dm^{-J}; \0}, \qquad \ell'=1, \ldots, \mphi \quad \mbox{and}\quad
\tilde \psi^{n'}:=\tilde \psi^{n';J,\0}_{\dm^{-J}; \0}, \qquad n'=1, \ldots, \mpsi.
\end{equation}
Now we prove that \eqref{bio:phi} and \eqref{bio:psi} must hold.
In fact, by the definition in \eqref{drw}, we have
\[
\la \phi^\ell_{\dm^J; \vk}, \tilde \phi^{\ell'}_{\dm^J; \vk'}\ra=
\la \phi^\ell_{\dm^J; \vk-\vk'}, \tilde \phi^{\ell'}_{\dm^J; \0}\ra
=\la \phi^\ell_{\dm^J; \vk-\vk'}, \tilde \phi^{\ell';J,\0}\ra
=\gd(\ell-\ell')\gd(\vk-\vk')
\]
and for $j\ge J$, noting that $\dm^{j-J}\dZ\subseteq \dZ$, we have
\[
\la \psi^n_{\dm^j; \vk}, \tilde \phi^{\ell'}_{\dm^J; \vk'}\ra
=\la \psi^n_{\dm^j; \vk-\dm^{j-J}\vk'}, \tilde \phi^{\ell'}_{\dm^J; \0}\ra=\la \psi^n_{\dm^j; \vk-\dm^{j-J}\vk'}, \tilde \phi^{\ell';J, \0}\ra=0.
\]
We also observe that
\begin{equation}\label{bio:aux}
\la \phi^\ell_{\dm^J; \vk}, \tilde \psi^{n'}_{\dm^J; \vk'}\ra
=\la \phi^\ell_{\dm^J; \vk-\vk'}, \tilde \psi^{n'}_{\dm^J; \0}\ra
=\la \phi^\ell_{\dm^J; \vk-\vk'}, \tilde \psi^{n'; J,\0}\ra=0.
\end{equation}

We now prove the rest of \eqref{bio:phi} and \eqref{bio:psi} by a similar argument as in \cite[Lemma~1]{Han:pams:2006}.
For $j'>J$, by \eqref{refeq:phi:psi}, we see that $\phi^\ell_{\dm^J;\vk}$ is an $l_2$-linear combination of $\phi^{L}_{\dm^{j'};\vm}$, $\vm\in \dZ$ and $L=1, \ldots, \mphi$.
By \eqref{bio:aux}, we have $\la \phi^{L}_{\dm^{j'};\vm}, \tilde \psi^{n'}_{\dm^{j'};\vk'}\ra=
\la \phi^{L}_{\dm^{J};\vm}, \tilde \psi^{n'}_{\dm^{J};\vk'}\ra=0$.
Therefore, $\la \phi^\ell_{\dm^J; \vk}, \tilde \psi^{n'}_{\dm^{j'};\vk'}\ra=0$.
Hence, \eqref{bio:phi} is verified.

To prove \eqref{bio:psi}, we consider two cases. If $j\ge j'\ge J$, then
\begin{align*}
\la \psi^n_{\dm^j;\vk}, \tilde \psi^{n'}_{\dm^{j'};\vk'}\ra
&=\la \psi^n_{\dm^{j-j'+J};\vk}, \tilde \psi^{n'}_{\dm^J;\vk'}\ra
=\la \psi^n_{\dm^{j-j'+J};\vk-\dm^{j-j'}\vk'}, \tilde \psi^{n'}_{\dm^J;\0}\ra\\
&=\la \psi^n_{\dm^{j-j'+J};\vk-\dm^{j-j'}\vk'}, \tilde \psi^{n';J,\0}\ra=\gd(j-j')\gd(\vk-\vk').
\end{align*}
If $j'>j\ge J$, then by \eqref{refeq:phi:psi}, $\psi^{n}_{\dm^j; \vk}$ is an $l_2$-linear combination of $\phi^L_{\dm^{j'};\vm}$, $\vm\in \dZ$ and $L=1, \ldots, \mphi$.
By \eqref{bio:aux}, we have $\la \phi^L_{\dm^{j'};\vm}, \tilde \psi^{n'}_{\dm^{j'};\vk'}\ra=\la \phi^L_{\dm^{J};\vm}, \tilde \psi^{n'}_{\dm^{J};\vk'}\ra=0$. Therefore, $\la \psi^n_{\dm^j;\vk}, \tilde \psi^{n'}_{\dm^{j'};\vk'}\ra=0$.
Hence, \eqref{bio:psi} is verified.

In conclusion, by the uniqueness of a dual Riesz basis of $\WS_J(\Phi; \Psi)$, we proved that $\tilde \phi^{\ell'}_{\dm^J;\vk'}=
\tilde \phi^{\ell';J,\vk'}$ and $\tilde \psi^{n'}_{\dm^j;\vk'}=\tilde \psi^{n'; j, \vk'}$ for all $j'\ge J$, $\vk'\in \dZ$ and $\ell'=1, \ldots, \mphi, n'=1,\ldots, \mpsi$.
Therefore, the pair $(\WS_J(\Phi;\Psi), \WS_J(\tilde \Phi; \tilde \Psi))$ is a pair of nonhomogeneous biorthogonal $\dm$-wavelet bases in $\dLp{2}$.
This completes the proof of (ii)$\Rightarrow$(i).
\end{proof}

We conclude this section by some remarks. Firstly, by the results in this section, we see that a nonhomogeneous wavelet system $\WS_J(\Phi; \Psi)$ in $\dLp{2}$ for a given integer $J$ (that is, at the coarsest scale level $J$) naturally leads to a sequence of nonhomogeneous wavelet systems $\WS_j(\Phi; \Psi)$ for all integers $j\in \Z$ while preserving almost all the properties of the system at $j=J$.
This is a fundamental property in wavelet analysis. In fact, a one-level fast wavelet transform is just a transform between the two sets of wavelet coefficients of a given function represented under two nonhomogeneous wavelet systems at two consecutive scale levels. Naturally, for a multi-level wavelet transform, there is an underlying sequence of wavelet systems at every scale level, instead of just one single wavelet system. For a given nonhomogeneous (stationary) wavelet system, since it naturally produces a sequence of nonhomogeneous wavelet systems, this allows us to study only one nonhomogeneous wavelet system instead of a sequence of nonhomogeneous wavelet systems. This desirable property of nonhomogeneous wavelet systems is not shared by homogeneous wavelet systems.
Secondly, by the results in this section, for any $d\times d$ real-valued expansive matrix $\dm$, we see that a nonhomogeneous wavelet system $\WS_J(\Phi; \Psi)$ in $\dLp{2}$ for a given integer $J$ naturally leads to a homogeneous wavelet system $\WS(\Psi)$. Therefore, a homogeneous wavelet system $\WS(\Psi)$ can be regarded as the limit of a sequence of nonhomogeneous wavelet systems $\WS_j(\Phi;\Psi)$ as $j\to -\infty$. Due to Theorem~\ref{thm:bw}, we see that a nonhomogeneous orthonormal wavelet basis has natural connections to refinable structures. This motivates us to introduce the notion of a homogeneous wavelet system with the quasi-refinable structure. For a given homogeneous wavelet system $\WS(\Psi)$ being a frame or a Riesz basis in $\dLp{2}$, we say that the homogeneous wavelet system $\WS(\Psi)$ has the quasi-refinable structure, if there exists a subset $\Phi$ of $\dLp{2}$ such that the nonhomogeneous wavelet system $\WS_0(\Phi;\Psi)$ is a frame or a Riesz basis in $\dLp{2}$ so that the homogeneous wavelet system $\WS(\Psi)$ is its limit system. Lastly, we mention that results in this section for nonhomogeneous wavelet systems in $\dLp{2}$ can be generalized to more general function spaces such as Sobolev spaces and Besov spaces. For the case of Sobolev spaces, see \cite[Theorem~7]{Han:wavelet} and \cite{HanShen:ca:2009} for more detail.

\section{Frequency-based Nonstationary Wavelet Systems in the Distribution Space}

To characterize nonhomogeneous dual or tight wavelet frames in $\dLp{2}$, we shall take a frequency-based approach by studying frequency-based nonstationary wavelet systems in the distribution space. More precisely, we shall introduce and characterize a pair of frequency-based nonstationary dual wavelet frames in the distribution space. As pointed out in \cite{Han:wavelet} for dimension one, such a notion allows us to completely separate the perfect reconstruction property of a wavelet system from its stability in various function spaces. Results in this section will serve as our basis to study nonhomogeneous and directional nonstationary tight wavelet frames in $\dLp{2}$ in the next section.

Following the standard notation, we denote by $\dD$ the linear space of all compactly supported $C^\infty$ (test) functions with the usual topology, and $\dDpr$ the linear space of all distributions, that is, $\dDpr$ is the dual space of $\dD$. By duality, it is easy to see that translation, dilation and modulation in \eqref{tdm} can be naturally extended to distributions in $\dDpr$.  For a tempered distribution $f$, by the definition of the notation $f_{U; \vk, \vn}$ in \eqref{tdm}, we have
\begin{equation}\label{ftdm}
\wh{f_{U; \vk, \vn}}=e^{-\iu \vk \cdot \vn} \hat f_{(U^T)^{-1}; -\vn, \vk} \quad \mbox{and}\quad \wh{f_{U; \vk}}=\hat f_{(U^T)^{-1}; \0, \vk}.
\end{equation}
In this paper, we shall use boldface letters to denote functions/distributions or sets of functions/distributions in the frequency domain.

By $\dlLp{p}$ we denote the linear space of all measurable functions $f$ such that $\int_K |f(x)|^p dx<\infty$ for every compact subset $K$ of $\dR$ with the usual modification for $p=\infty$. Note that $\dlLp{1}$ is just the set of all locally integrable functions that can be globally identified as distributions, that is, $\dlLp{1}\subseteq \dDpr$.
For $\ff\in \dD$ and $\fpsi\in \dlLp{1}$, we shall use the following paring
\begin{equation}\label{pair}
\la \ff, \fpsi\ra:=\int_{\dR} \ff(\xi) \overline{\fpsi(\xi)} d\xi \quad \mbox{and}\quad
\la \fpsi, \ff\ra:=\ol{\la \ff, \fpsi\ra}=
\int_{\dR} \fpsi(\xi) \overline{\ff(\xi)} d\xi.
\end{equation}
When $\ff\in \dD$ and $\fpsi\in \mathscr{D}'(\dR)$, the duality pairings $\la \ff, \fpsi\ra$ and $\la \fpsi, \ff\ra$ are understood similarly as $\la \ff, \fpsi\ra:=\overline{\la \fpsi, \ff\ra}:=\overline{\fpsi(\overline{\ff})}$.

Let $J$ be an integer and $\dn_j, j\ge J$ be $d\times d$ real-valued invertible matrices.
Let $\fPhi$ and $\fPsi_j$, $j\ge J$ be subsets of distributions. {\it A frequency-based nonstationary wavelet system} is defined to be
\begin{equation}\label{fnws}
\FWS_J(\fPhi; \{\fPsi_j\}_{j=J}^\infty)=
\{ \fphi_{\dn_J; \0, \vk} \; : \; \vk\in \dZ, \fphi\in \fPhi\} \cup \bigcup_{j=J}^\infty \{ \fpsi_{\dn_j; \0, \vk} \; : \; \vk\in \dZ, \fpsi\in \fPsi_j\}.
\end{equation}
For the particular case $\dn_j=\dn^j$
and $\fPsi_j=\fPsi$ for all $j\ge J$, a frequency-based nonstationary wavelet system in \eqref{fnws} becomes {\it a frequency-based nonhomogeneous (stationary) $\dn$-wavelet system}:
\begin{equation}\label{fws}
\FWS_J(\fPhi; \fPsi)=
\{ \fphi_{\dn^J; \0, \vk} \; : \; \vk\in \dZ, \fphi\in \fPhi\}
\cup \{ \fpsi_{\dn^j; \0, \vk} \; : \; j\ge J, \vk\in \dZ, \fpsi\in \fPsi\}.
\end{equation}
For a nonhomogeneous $\dm$-wavelet system $\WS_J(\Phi; \Psi)$ such that all the generators in $\Phi$ and $\Psi$ are tempered distributions, by \eqref{ftdm}, the image of the nonhomogeneous $\dm$-wavelet system $\WS_J(\Phi; \Psi)$ under the Fourier transform simply becomes
the frequency-based nonhomogeneous $(\dm^T)^{-1}$-wavelet system
$\FWS_J(\fPhi; \fPsi)$, where
\begin{equation}\label{F:Phi:Psi}
\fPhi=\{ \hat{\phi}\; : \; \phi\in \Phi\},\qquad
\fPsi=\{ \hat{\psi}\; : \; \psi\in \Psi\}.
\end{equation}
For analysis of wavelets and framelets, as argued in \cite{Han:wavelet} for dimension one, it is often easier to work with frequency-based wavelet systems $\FWS_J
(\fPhi; \fPsi)$ instead of space/time-based wavelet systems
$\WS_J(\Phi; \Psi)$, though both are equivalent to each other under the framework of tempered distributions. Since we consider frequency-based wavelets and framelets in the distribution space $\dDpr$, it is natural for us to consider $\FWS_J(\fPhi; \fPsi)\subseteq \dDpr$.

Let $\dn_j, j\ge J$ be $d\times d$ real-valued invertible matrices. Let
\begin{equation}\label{fPhi}
\fPhi=\{\fphi^1, \ldots, \fphi^\mphi\} \quad \mbox{and}\quad
\tilde \fPhi=\{\tilde \fphi^1, \ldots, \tilde \fphi^\mphi\}
\end{equation}
and
\begin{equation}\label{fPsij}
\fPsi_j=\{\fpsi^{j,1}, \ldots, \fpsi^{j,\mpsi_j}\} \quad \mbox{and}\quad
\tilde \fPsi_j=\{\tilde \fpsi^{j,1}, \ldots, \tilde \fpsi^{j,\mpsi_j}\}
\end{equation}
be subsets of $\dDpr$ for all integers $j\ge J$, where $\mphi, \mpsi\in \N \cup \{0,+\infty\}$. Let $\FWS_J(\fPhi; \{\fPsi_j\}_{j=J}^\infty)$ be defined in \eqref{fws} and $\FWS_J(\tilde \fPhi; \{\tilde \fPsi_j\}_{j=J}^\infty)$ be defined similarly. Generalizing the notion in \cite{Han:wavelet} from dimension one to high dimensions, we say that the pair
\begin{equation}\label{fnws:pair}
(\FWS_J (\fPhi; \{\fPsi_j \}_{j=J}^\infty),
\FWS_J (\tilde\fPhi; \{\tilde \fPsi_j\}_{j=J}^\infty))
\end{equation}
is {\it a pair of frequency-based nonstationary dual wavelet frames in the distribution space $\dDpr$} if
the following identity holds
\begin{equation}\label{fnws:dwf}
\sum_{\ell=1}^{\mphi} \sum_{\vk\in \dZ} \la \ff, \fphi^{\ell}_{\dn_J; \0, \vk}\ra
\la \tilde \fphi^{\ell}_{\dn_J; \0, \vk}, \fg\ra
+
\sum_{j=J}^{\infty} \sum_{\ell=1}^{\mpsi_j} \sum_{\vk\in \dZ}
\la \ff, \fpsi^{j,\ell}_{\dn_j; \0, \vk}\ra
\la \tilde \fpsi^{j,\ell}_{\dn_j; \0, \vk}, \fg\ra
=(2\pi)^d \la \ff, \fg\ra
\end{equation}
for all $\ff, \fg\in \dD$, where the infinite series in \eqref{fnws:dwf} converge in the following sense
\begin{enumerate}
\item[{\rm(i)}] for every $\ff, \fg\in \dD$, all the series
\begin{equation}\label{series}
\sum_{\ell=1}^\mphi \sum_{\vk\in \dZ} \la \ff, \fphi^{\ell}_{\dn_J; \0, \vk}\ra
\la \tilde \fphi^{\ell}_{\dn_J; \0, \vk}, \fg\ra
\quad \mbox{and}\quad \sum_{\ell=1}^{\mpsi_j} \sum_{\vk\in \dZ}
\la \ff, \fpsi^{j,\ell}_{\dn_j; \0, \vk}\ra
\la \tilde \fpsi^{j,\ell}_{\dn_j; \0, \vk}, \fg\ra
\end{equation}
converge absolutely for all integers $j\ge J$;

\item[{\rm(ii)}] for every $\ff, \fg\in \dD$, the following limit exists and
\begin{equation}\label{dwf:lim}
\lim_{J'\to +\infty} \Big(\sum_{\ell=1}^{\mphi} \sum_{\vk\in \dZ} \la \ff, \fphi^{\ell}_{\dn_J; \0, \vk}\ra
\la \tilde \fphi^{\ell}_{\dn_J; \0, \vk}, \fg\ra
+
\sum_{j=J}^{J'-1} \sum_{\ell=1}^{\mpsi_j} \sum_{\vk\in \dZ}
\la \ff, \fpsi^{j,\ell}_{\dn_j; \0, \vk}\ra
\la \tilde \fpsi^{j,\ell}_{\dn_j; \0, \vk}, \fg\ra\Big)
=(2\pi)^d \la \ff, \fg\ra.
\end{equation}
\end{enumerate}

We say that the pair in \eqref{fnws:pair} is {\it a pair of frequency-based nonstationary dual wavelet frames in $\dLp{2}$} if (i) all elements in the two systems of the pair belong to $\dLp{2}$, (ii) each system in the pair is a frame in $\dLp{2}$, and (iii) \eqref{fnws:dwf} holds for all $\ff, \fg\in \dLp{2}$
with the series converging absolutely. It is straightforward to see that a space/time-based pair $(\WS_J(\Phi; \Psi), \WS_J(\tilde \Phi;\tilde \Psi))$ is a pair of nonhomogeneous dual $\dm$-wavelet frames in $\dLp{2}$ if and only if the frequency-based pair
$(\FWS_J(\fPhi; \fPsi), \FWS_J(\tilde \fPhi; \tilde \fPsi))$
is a pair of frequency-based nonhomogeneous dual $(\dm^T)^{-1}$-wavelet frames in $\dLp{2}$, where $\fPhi, \fPsi$ are defined in \eqref{F:Phi:Psi} and
\begin{equation}\label{F:tilde:Phi:Psi}
\tilde \fPhi=\{ \hat{\tilde\phi}\; : \; \tilde \phi\in \tilde \Phi\},\qquad
\tilde \fPsi=\{ \hat{\tilde \psi}\; : \; \tilde \psi\in \tilde \Psi\}.
\end{equation}

Before we study and characterize a pair of frequency-based nonstationary dual wavelet frames in the distribution space $\dDpr$, by the following result, we see that the above notion allows us to completely separate the perfect reconstruction property of a wavelet system from its stability in the function space $\dLp{2}$.

\begin{theorem}\label{thm:fndwf}
Let $\dn_j, j\ge J$ be $d\times d$ real-valued invertible matrices. Let $\fPhi, \tilde \fPhi$ and $\fPsi_j, \tilde \fPsi_j$ be at most countable subsets of distributions on $\dR$ for all integers $j\ge J$. Then the pair in \eqref{fnws:pair} is a pair of frequency-based nonstationary dual wavelet frames in $\dLp{2}$ if and only if
\begin{enumerate}
\item[{\rm(i)}] there exists a positive constant $C$ such that
\begin{align}
&\sum_{\fphi\in \fPhi} \sum_{\vk\in \dZ} |\la \ff, \fphi_{\dn_J; \0, \vk}\ra|^2+\sum_{j=J}^\infty \sum_{\fpsi\in \fPsi_j} \sum_{\vk\in \dZ} |\la \ff, \fpsi_{\dn_j; \0, \vk}\ra|^2\le C\|\ff\|^2_{\dLp{2}}, \qquad \forall\; \ff\in \dD, \label{primal:bound}\\
&\sum_{\tilde \fphi\in \tilde \fPhi} \sum_{\vk\in \dZ} |\la \fg, \tilde\fphi_{\dn_J; \0, \vk}\ra|^2+\sum_{j=J}^\infty \sum_{\tilde \fpsi\in \tilde \fPsi_j} \sum_{\vk\in \dZ} |\la \fg, \tilde\fpsi_{\dn_j; \0, \vk}\ra|^2\le C\|\fg\|^2_{\dLp{2}}, \qquad \forall\; \fg\in \dD; \label{dual:bound}
\end{align}
\item[{\rm(ii)}] the pair in \eqref{fnws:pair} is a pair of frequency-based nonstationary dual wavelet frames in the distribution space $\dDpr$.
\end{enumerate}
\end{theorem}

\begin{proof} The necessity part is trivial, since $\dD\subseteq \dLp{2}$. It suffices to prove the sufficiency part. By \eqref{primal:bound}, we have $|\la \ff, \fphi_{\dn_J;\0,\0}\ra|^2\le C \|\ff\|_{\dLp{2}}^2$ for all $\ff\in \dD$. Therefore, we see that $\la \cdot, \fphi_{\dn_J;\0,\0}\ra$ can be extended into a bounded linear functional on $\dLp{2}$, from which we conclude that $\fphi_{\dn_J;\0,\0}$ can be identified with a function in $\dLp{2}$. Since $\dn_J$ is invertible, we have $\fphi\in \dLp{2}$ for all $\fphi\in \fPhi$. By the same argument, we see that both systems in \eqref{fnws:pair} have all their elements in $\dLp{2}$.
Consequently, since $\dD$ is dense in $\dLp{2}$ and both systems are countable sets, by a standard argument using density, we see that \eqref{primal:bound} and \eqref{dual:bound} hold for all $\ff, \fg\in \dLp{2}$.
By item (ii), \eqref{fnws:dwf} holds for $\ff, \fg\in \dD$.
Since \eqref{primal:bound} and \eqref{dual:bound} hold for all $\ff, \fg\in \dLp{2}$, using Cauchy-Schwarz inequality,
we see that \eqref{fnws:dwf} holds for all $\ff, \fg\in \dLp{2}$ with the series converging absolutely.
Now by \eqref{fnws:dwf} and using Cauchy-Schwarz inequality, it follows from \eqref{primal:bound} and \eqref{dual:bound} that the left-hand sides of the inequalities in \eqref{primal:bound} and \eqref{dual:bound} are no less than $(2\pi)^d C^{-1}\|\ff\|_{\dLp{2}}^2$ and $(2\pi)^d C^{-1}\|\fg\|_{\dLp{2}}^2$, respectively. This completes the proof of the sufficiency part.
\end{proof}

A pair of frequency-based (or space/time-based) nonstationary dual wavelet frames can be generalized from the function space $\dLp{2}$ to a pair of dual function spaces $(\mathscr{B}, \mathscr{B}')$ where $\mathscr{B}$ is a Banach function space contained in $\dDpr$ and $\mathscr{B}'$ is its dual space. For example, $\mathscr{B}$ can be a Sobolev or Besov space. The result in Theorem~\ref{thm:fndwf} holds for such more general function spaces by replacing item (i) with suitable stability condition, more precisely, the boundedness of the primal system in the function space $\mathscr{B}$ and the boundedness of the dual system in the function space $\mathscr{B}'$. That many classical function spaces can often be characterized by the wavelet coefficient sequences largely lies in the fact that we have the perfect reconstruction property in item (ii) while the two wavelet systems are often bounded in many function spaces. For the case of Sobolev spaces, see \cite[Theorem~7]{Han:wavelet} and \cite{HanShen:ca:2009} for more details on pairs of dual wavelet frames in function spaces other than $\dLp{2}$.

To characterize a pair of frequency-based nonstationary dual wavelet frames in the distribution space,
let us first look at the absolute convergence of the series in \eqref{series}. By the following result which generalizes \cite[Lemma~3]{Han:wavelet},
we see that the absolute convergence of the series in \eqref{series} holds under a very mild condition.

\begin{lemma}\label{lem:converg} Let $U$ be a $d\times d$ real-valued invertible matrix and $\fpsi, \tilde \fpsi\in \dlLp{2}$. Then for all $\ff,\fg\in \dD$,
\begin{equation}\label{parseval}
\sum_{\vk\in \dZ} \la \ff, \fpsi_{U; \0, \vk}\ra \la \tilde \fpsi_{U;\0,\vk}, \fg\ra=(2\pi)^d \int_{\dR} \sum_{\vk\in \dZ} \ff(\xi) \ol{\fg(\xi+2\pi U^{-1}\vk)}\, \ol{\fpsi(U\xi)} \tilde \fpsi(U\xi+2\pi \vk)\, d\xi
\end{equation}
with the series on the left-hand side converging absolutely. Note that the infinite sum on the right-hand side of \eqref{parseval} is in fact finite.
\end{lemma}

\begin{proof} Define two $2\pi\dZ$-periodic functions $\fh$ and $\tilde \fh$ as follows:
\begin{equation}\label{h}
\fh(\xi):=\sum_{\vn\in \dZ} \ff(U^{-1}(\xi+2\pi \vn))\ol{\fpsi(\xi+2\pi \vn)} \quad \mbox{and}\quad
\tilde \fh(\xi):=\sum_{\vn\in \dZ} \fg(U^{-1}(\xi+2\pi \vn))\ol{\tilde \fpsi(\xi+2\pi \vn)}.
\end{equation}
Since $\ff, \fg\in \dD$, the infinite sums in \eqref{h} are in fact finite for $\xi\in [-\pi, \pi]^d$. Moreover, by $\fpsi, \tilde \fpsi \in \dlLp{2}$, we deduce that $\fh, \tilde \fh\in \dTLp{2}$. By calculation, we have
\[
\int_{[-\pi,\pi)^d} \fh(\xi) e^{\iu \vk\cdot \xi} d\xi=|\det U|^{1/2} \la \ff, \fpsi_{U; \0, \vk}\ra \quad \mbox{and}\quad
\int_{[-\pi,\pi)^d} \tilde \fh(\xi) e^{\iu \vk\cdot \xi} d\xi=|\det U|^{1/2} \la \fg, \tilde \fpsi_{U; \0, \vk}\ra.
\]
Now by Parsevel identity for periodic functions in $\dTLp{2}$, one can easily deduce that the left-hand side of \eqref{parseval} is equal to
\begin{align*}
(2\pi)^d |\det U|^{-1} \int_{[-\pi, \pi)^d} \fh(\xi) \ol{\tilde \fh(\xi)}d\xi&=(2\pi)^d |\det U|^{-1}\int_{\dR} \ff(U^{-1}\xi)\ol{\fpsi(\xi)} \ol{\tilde \fh(\xi)}d\xi\\
&= (2\pi)^d \int_{\dR} \ff(\xi)\ol{\fpsi(U\xi)} \ol{\tilde \fh(U\xi)}d\xi,
\end{align*}
from which we see that \eqref{parseval} holds. Since $\ff, \fg$ are compactly supported and $U$ is invertible, $\ff(\xi)\ol{\fg(\xi+2\pi U^{-1}\vk)}=0$ for all $\xi\in \dR$ provided that $\|\vk\|$ is large enough. Hence, the infinite sum on the right-hand side of \eqref{parseval} is in fact finite.
\end{proof}

The assumption of membership in $\dlLp{2}$ is only used in this paper to guarantee \eqref{parseval} with the series on the left-hand side of \eqref{parseval} converging absolutely.

The following result characterizes a pair of frequency-based nonstationary dual wavelet frames in the distribution space.

\begin{theorem}\label{thm:fndwf:criterion}
Let $J$ be an integer. Let $\dn_j, j\ge J$ be $d\times d$ real-valued invertible matrices such that
\begin{equation}\label{cond:dn}
\gL:=\cup_{j=J}^\infty [\dn_j^{-1}\dZ] \; \mbox{has no accumulation point}\quad\mbox{and}\quad
\lim_{j\to +\infty} \|\dn_j \vk\|=0\qquad \forall\; \vk\in \gL.
\end{equation}
Let $\fPhi, \tilde \fPhi$ in \eqref{fPhi} and $\fPsi_j, \tilde \fPsi_j$ in \eqref{fPsij} be finite subsets of $\lLp{2}$ for all integers $j\ge J$. Then the pair in \eqref{fnws:pair} is a pair of frequency-based nonstationary dual wavelet frames in the distribution space $\dDpr$ if and only if
\begin{equation}\label{basic:cond:1}
\lim_{J'\to +\infty} \Big \la \I_{\fPhi}^{\0}(\dn_J\cdot)+\sum_{j=J}^{J'-1} \I_{\fPsi_j}^{\0}(\dn_j\cdot), \fh\Big\ra=\la 1, \fh\ra \qquad \forall \; \fh\in \dD
\end{equation}
and
\begin{equation}\label{basic:cond:2}
\I_{\fPhi}^{\dn_J \vk}(\dn_J\xi)+\sum_{j=J}^\infty \I_{\fPsi_j}^{\dn_j \vk} (\dn_j\xi)=0, \qquad a.e.\, \xi\in \dR, \vk\in \gL\bs \{\0\},
\end{equation}
(the infinite sums in \eqref{basic:cond:2} are in fact finite.) where
\begin{align}
&\I_{\fPhi}^{\vk}(\xi):=\sum_{\ell=1}^{\mphi} \ol{\fphi^{\ell}(\xi)}\tilde \fphi^{\ell}(\xi+2\pi \vk), \quad \vk\in \dZ \quad \hbox{and}\quad
\I_{\fPhi}^{\vk}:= 0 \qquad \vk \in \dR \bs \dZ,\label{I:Phi}\\
&\I_{\fPsi_j}^{\vk}(\xi):=\sum_{\ell=1}^{\mpsi_j} \ol{\fpsi^{j,\ell}(\xi)}\tilde \fpsi^{j,\ell}(\xi+2\pi \vk), \quad \vk\in \dZ \quad \hbox{and}\quad
\I_{\fPsi_j}^{\vk}:= 0 \qquad \vk \in \dR \bs \dZ.\label{I:Psij}
\end{align}
\end{theorem}

\begin{proof} For $\ff,\fg\in \dD$, denote
\begin{equation}\label{Sfg}
S_J^{J'}(\ff,\fg):=\sum_{\ell=1}^{\mphi} \sum_{\vk\in \dZ} \la \ff, \fphi^{\ell}_{\dn_J; \0, \vk}\ra
\la \tilde \fphi^{\ell}_{\dn_J; \0, \vk}, \fg\ra
+
\sum_{j=J}^{J'-1} \sum_{\ell=1}^{\mpsi_j} \sum_{\vk\in \dZ}
\la \ff, \fpsi^{j,\ell}_{\dn_j; \0, \vk}\ra
\la \tilde \fpsi^{j,\ell}_{\dn_j; \0, \vk}, \fg\ra.
\end{equation}
By Lemma~\ref{lem:converg}, we have
\begin{equation}\label{S:eq1}
S_J^{J'}(\ff, \fg)=(2\pi)^d \int_{\dR} \sum_{\vk\in \gL}
\ff(\xi)\ol{\fg(\xi+2\pi \vk)}\Big(\I_{\fPhi}^{\dn_J\vk}(\dn_J\xi)+\sum_{j=J}^{J'-1} \I_{\fPsi_j}^{\dn_j\vk}(\dn_j\xi)\Big)d\xi.
\end{equation}
On the one hand, since $\ff$ and $\fg$ are compactly supported,
there exists a positive constant $c$ such that $\ff(\xi)\ol{\fg(\xi+2\pi \vk)}=0$ for all $\xi\in \dR$ and $|\vk|\ge c$. On the other hand, by our assumption in \eqref{cond:dn},
$\{ \vk\in \gL\; : \; |\vk|<c\}$ is a finite set and therefore, there exists a positive integer $J''$ such that $\dn_j \vk\not \in \dZ$ for all $j\ge J''$ and $\vk\in \gL\bs \{\0\}$ with $|\vk|<c$.
That is, under the assumption in \eqref{cond:dn}, for all $J'\ge J''$, \eqref{S:eq1} becomes
\begin{equation}\label{S:eq1:0}
\begin{split}
S_J^{J'}(\ff, \fg)=&(2\pi)^d \int_{\dR}
\ff(\xi)\ol{\fg(\xi)}\Big(\I_{\fPhi}^{\0}(\dn_J\xi)+\sum_{j=J}^{J'-1} \I_{\fPsi_j}^{\0}(\dn_j\xi)\Big) d\xi\\
&+(2\pi)^d \int_{\dR} \sum_{\vk\in \gL\bs \{0\}}
\ff(\xi)\ol{\fg(\xi+2\pi \vk)}\Big( \I_{\fPhi}^{\dn_J\vk}(\dn_J\xi)+\sum_{j=J}^{\infty} \I_{\fPsi_j}^{\dn_j\vk}(\dn_j\xi)\Big) d\xi.
\end{split}
\end{equation}
If \eqref{basic:cond:1} holds, then we deduce from \eqref{S:eq1:0} that
\begin{equation}\label{S:eq2}
S_J^{J'}(\ff, \fg)=(2\pi)^d \int_{\dR}
\ff(\xi)\ol{\fg(\xi)}\Big( \I_{\fPhi}^{\0}(\dn_J\xi)+\sum_{j=J}^{J'-1} \I_{\fPsi_j}^{\0}(\dn_j\xi)\Big) d\xi.
\end{equation}
Now it follows from \eqref{basic:cond:1} that
\begin{equation}\label{S:eq3}
\lim_{J'\to +\infty} S_J^{J'}(\ff, \fg)=(2\pi)^d \la \ff, \fg\ra,\qquad \ff, \fg\in \dD.
\end{equation}
Therefore, the pair in \eqref{fnws:pair} is a pair of frequency-based nonstationary dual wavelet frames in the distribution space $\dDpr$.

Conversely, \eqref{S:eq3} holds. By our assumption in \eqref{cond:dn}, \eqref{S:eq1:0} holds for all $J'\ge J''$. Note that the set $\gL$ is discrete and closed, that is, for any $\vk \in \gL$, $\gep_{\vk}:=\inf_{y\in \gL \bs \{\vk\}}\| y-\vk\|/2>0$. For any $\vk\in \gL \bs \{0\}$ and $\xi_0\in \dR$, we deduce from \eqref{S:eq1:0} and \eqref{S:eq3} that for $J'\ge J''$,
\begin{equation}\label{S:eq4}
(2\pi)^d \int_{\dR}
 \ff(\xi)\ol{\fg(\xi+2\pi \vk)}\Big(\I_{\fPhi}^{\dn_J \vk}(\dn_J\xi)+\sum_{j=J}^{\infty} \I_{\fPsi_j}^{\dn_j \vk}(\dn_j\xi)\Big) d\xi=S_J^{J'}(\ff, \fg)=(2\pi)^d \la \ff, \fg\ra=0
\end{equation}
for all $\ff, \fg\in \dD$ such that $\mbox{supp}\, \ff\subseteq B_{\gep_\vk}(\xi_0)$ and $\mbox{supp}\, \fg\subseteq B_{\gep_\vk}(\xi_0-2\pi \vk)$, where $B_{\gep_\vk}(\xi_0):=\{ \xi\in \dR\, : \, \|\xi-\xi_0\|\le \gep_k\}$ is the ball with center $\xi_0$ and radius $\gep_\vk$.
Now we can easily deduce from the above relation in \eqref{S:eq4} that identity in \eqref{basic:cond:2} holds for almost every $\xi\in B_{\gep_\vk}(\xi_0)$. Since $\xi_0$ is arbitrary, we conclude that \eqref{basic:cond:2} holds. By \eqref{basic:cond:2} and \eqref{S:eq1:0}, we see that \eqref{S:eq2} holds, from which we see that \eqref{basic:cond:1} holds.
\end{proof}

As we shall see in the proof of Corollary~\ref{cor:ndwf}, the condition in \eqref{cond:dn} is automatically satisfied if $\dn_j=\dn^j$ for all $j\ge J$ and all the eigenvalues of $\dn$ are less than one in modulus.
Using a more technical argument, we also see that Theorem~\ref{thm:fndwf:criterion} (as well as other results in this paper) still holds if the assumption, that $\fPhi, \tilde \fPhi, \fPsi_j, \tilde \fPsi_j$ are finite subsets of $\dlLp{2}$ for all $j\ge J$ in Theorem~\ref{thm:fndwf:criterion}, is replaced, for example, by the assumption
\[
\sum_{\fphi\in \fPhi} |\fphi(\cdot)|^2, \qquad \sum_{\tilde \fphi\in \tilde \fPhi} |\tilde \fphi(\cdot)|^2, \qquad  \sum_{\fpsi\in \fPsi_j} |\fpsi(\cdot)|^2, \qquad \sum_{\tilde \fpsi\in \tilde \fPsi_j} |\tilde \fpsi(\cdot)|^2, \qquad j\ge J
\]
belong to $\dlLp{2}$.

The following result is a direct consequence of Theorems~\ref{thm:fndwf} and~\ref{thm:fndwf:criterion}.

\begin{cor}\label{cor:fntwf}
Let $\dn_j, j\ge J$ be $d\times d$ real-valued invertible matrices. Let $\fPhi$ and $\fPsi_j$ be subsets of $\dDpr$ for all $j\ge J$. Then the following statements are equivalent to each other:
\begin{enumerate}
\item[{(i)}] $\FWS_J(\fPhi; \{\fPsi_j\}_{j=J}^\infty)$ is a frequency-based nonstationary tight wavelet frame in $\dLp{2}$, that is, $\fPhi, \fPsi_j \subseteq \dLp{2}$ for all $j\ge J$ and
\begin{equation}\label{fntwf}
\sum_{\fphi\in \fPhi} \sum_{\vk\in \dZ} |\la \ff, \fphi_{\dn_J; \0, \vk}\ra |^2+
\sum_{j=J}^{\infty} \sum_{\fpsi\in \fPsi_j} \sum_{\vk\in \dZ}
|\la \ff, \fpsi_{\dn_j; \0, \vk}\ra|^2
=(2\pi)^d \|\ff\|^2_{\dLp{2}}, \;\; \forall\; \ff\in \dLp{2};
\end{equation}
 \item[{(ii)}] $(\FWS_J(\fPhi; \{\fPsi_j\}_{j=J}^\infty), \FWS_J(\fPhi; \{\fPsi_j\}_{j=J}^\infty))$ is
a pair of frequency-based nonstationary dual wavelet frames in the distribution space $\dDpr$.
\end{enumerate}
If in addition \eqref{cond:dn} holds and all $\fPhi, \fPsi_j, j\ge J$ are finite subsets of $\dlLp{2}$, then any of the above statements is also equivalent to
\begin{enumerate}
\item[{(iii)}] \eqref{basic:cond:1} and \eqref{basic:cond:2} are satisfied with $\tilde \fPhi=\fPhi$ and $\tilde \fPsi_j=\fPsi_j$ for all integers $j\ge J$.
\end{enumerate}
\end{cor}

\begin{proof} By Theorem~\ref{thm:fndwf}, (i)$\Rightarrow$(ii). To see (ii)$\Rightarrow$(i),
since  $(\FWS_J(\fPhi; \{\fPsi_j\}_{j=J}^\infty), \FWS_J(\fPhi; \{\fPsi_j\}_{j=J}^\infty))$ is
a pair of frequency-based nonstationary dual wavelet frames in the distribution space, by definition, we can easily deduce that
\eqref{primal:bound} must hold with $C=(2\pi)^d$. Now by Theorem~\ref{thm:fndwf}, (ii)$\Rightarrow$(i). The equivalence between items (ii) and (iii) follows directly from
Theorem~\ref{thm:fndwf:criterion}.
\end{proof}

We have the following result on a sequence of pairs of frequency-based nonstationary dual wavelet frames in the distribution space.

\begin{theorem}\label{thm:fnws:special}
Let $J_0$ be an integer.  Let $\dn_j, j\ge J_0$ be $d\times d$ real-valued invertible matrices. Let
\begin{equation}\label{fPhij}
\fPhi_j=\{\fphi^{j,1}, \ldots, \fphi^{j,\mphi_j}\},\qquad
\tilde \fPhi_j=\{\tilde\fphi^{j,1}, \ldots, \tilde\fphi^{j,\mphi_j}\}
\end{equation}
and $\fPsi_j, \tilde \fPsi_j$ in \eqref{fPsij} be subsets of $\dDpr$.  Then the following statements are equivalent:
\begin{enumerate}
\item[{\rm(i)}] the pair $(\FWS_J(\fPhi_J; \{\fPsi_j\}_{j=J}^\infty), \FWS_J(\tilde\fPhi_J; \{\tilde\fPsi_j\}_{j=J}^\infty))$ is a pair of frequency-based nonstationary dual wavelet frames in the distribution space $\dDpr$ for every integer $J\ge J_0$;
\item[{\rm(ii)}] the following identities hold:
\begin{equation}\label{phi:to:limit}
\lim_{j\to +\infty} \sum_{\ell=1}^{\mphi_j} \sum_{\vk\in \dZ} \la \ff, \fphi^{j,\ell}_{\dn_j; \0, \vk}\ra
\la \tilde \fphi^{j,\ell}_{\dn_j; \0, \vk}, \fg\ra=(2\pi)^d\la \ff, \fg\ra, \qquad \forall\; \ff, \fg\in \dD
\end{equation}
and
\begin{equation}\label{mra}
\begin{split}
&\sum_{\ell=1}^{\mphi_j} \sum_{\vk\in \dZ} \la \ff, \fphi^{j,\ell}_{\dn_j; \0, \vk}\ra
\la \tilde \fphi^{j,\ell}_{\dn_j; \0, \vk}, \fg\ra
+\sum_{\ell=1}^{\mpsi_j} \sum_{\vk\in \dZ}
\la \ff, \fpsi^{j,\ell}_{\dn_j; \0, \vk}\ra
\la \tilde \fpsi^{j,\ell}_{\dn_j; \0, \vk}, \fg\ra\\
&\qquad\qquad
=\sum_{\ell=1}^{\mphi_{j+1}} \sum_{\vk\in \dZ} \la \ff, \fphi^{j+1,\ell}_{\dn_{j+1}; \0, \vk}\ra
\la \tilde \fphi^{j+1,\ell}_{\dn_{j+1}; \0, \vk}, \fg\ra, \qquad \forall\;  \ff, \fg\in \dD, j\ge J_0
\end{split}
\end{equation}
with all the series converging absolutely.
\end{enumerate}
If in addition \eqref{cond:dn} holds and all $\fPhi_j, \tilde \fPhi_j, \fPsi_j, \tilde \fPsi_j, j\ge J_0$ are finite subsets of $\dlLp{2}$, then any of the above statements is also equivalent to
\begin{enumerate}
\item[{\rm(iii)}] the following identities hold:
\begin{equation}\label{phi:to1}
\lim_{j\to +\infty} \Big \la \sum_{\ell=1}^{\mphi_j} \ol{\fphi^{j,\ell}(\dn_j\cdot)}\tilde \fphi^{j,\ell}(\dn_j\cdot), \fh\Big \ra=\la 1, \fh\ra, \qquad \forall\; \fh\in \dD
\end{equation}
and for all integers $j\ge J_0$,
\begin{equation}\label{cond:mra}
\I_{\fPhi_j}^{\dn_j \vk}(\dn_j\xi)+
\I_{\fPsi_j}^{\dn_j \vk}(\dn_j\xi)=\I_{\fPhi_{j+1}}^{\dn_{j+1} \vk}(\dn_{j+1}\xi), \;\; a.e.\; \xi\in \dR, \vk\in [\dn_j^{-1}\dZ]\cup [\dn_{j+1}^{-1}\dZ],
\end{equation}
where $\I_{\fPsi_j}^{\vk}, \vk\in \dR$ are defined in \eqref{I:Psij} and
\begin{equation}\label{I:Phij}
\I_{\fPhi_j}^{\vk}(\xi):=\sum_{\ell=1}^{\mphi_j} \ol{\fphi^{j,\ell}(\xi)}\tilde \fphi^{j,\ell}(\xi+2\pi \vk), \quad \vk\in \dZ \quad \hbox{and}\quad
\I_{\fPhi_j}^{\vk}:= 0 \qquad \vk \in \dR \bs \dZ.
\end{equation}
\end{enumerate}
\end{theorem}

\begin{proof} (i)$\Rightarrow$(ii). Considering the difference between two pairs at consecutive levels $j$ and $j+1$, from \eqref{dwf:lim} we see that \eqref{mra} holds. Now by \eqref{mra}, it is straightforward to see that
\begin{equation}\label{quasirefinable}
\begin{split}
\sum_{\ell=1}^{\mphi_J} \sum_{\vk\in \dZ} \la \ff, \fphi^{J,\ell}_{\dn_J;\0, \vk}\ra \la \tilde \fphi^{J,\ell}_{\dn_J;\0,\vk}, \fg\ra &+ \sum_{j=J}^{J'-1} \sum_{\ell=1}^{\mpsi_j} \sum_{\vk\in \dZ} \la \ff, \fpsi^{j,\ell}_{\dn_j;\0, \vk}\ra \la \tilde \fpsi^{j,\ell}_{\dn_j;\0,\vk}, \fg\ra\\
&=
\sum_{\ell=1}^{\mphi_{J'}} \sum_{\vk\in \dZ} \la \ff, \fphi^{J',\ell}_{\dn_{J'};\0, \vk}\ra \la \tilde \fphi^{J',\ell}_{\dn_{J'};\0,\vk}, \fg\ra,\qquad \ff, \fg\in \dD, J'\ge J.
\end{split}
\end{equation}
Therefore, \eqref{phi:to:limit} holds.

(ii)$\Rightarrow$(i). Since all the series in \eqref{mra} converge absolutely, all the series in \eqref{series} converge absolutely. To see item (i), by \eqref{mra}, we deduce that \eqref{quasirefinable} holds.
Now \eqref{dwf:lim} with $\fPhi=\fPhi_J$ and $\tilde \fPhi=\tilde \fPhi_J$ follows directly from \eqref{phi:to:limit} and \eqref{quasirefinable}.

(ii)$\Leftrightarrow$(iii). By Lemma~\ref{lem:converg}, we see that \eqref{mra} is equivalent to
\begin{equation}\label{mra:identity}
\int_{\dR} \sum_{\vk\in \gL_j} \ff(\xi)\ol{\fg(\xi+2\pi \vk)} \Big(\I_{\fPhi_j}^{\dn_j\vk}(\dn_j\xi)+\I_{\fPsi_j}^{\dn_j\vk}(\dn_j\xi)
-\I_{\fPhi_{j+1}}^{\dn_{j+1}\vk}(\dn_{j+1}\xi)\Big)d\xi=0,
\end{equation}
where $\gL_j:=[\dn_j^{-1}\dZ]\cup [\dn_{j+1}^{-1}\dZ]$.
Since $\gL_j$ is discrete, by the same argument as in Theorem~\ref{thm:fndwf:criterion}, we see that
\eqref{mra:identity} is equivalent to \eqref{cond:mra}.

By Lemma~\ref{lem:converg}, we see that \eqref{phi:to:limit} is equivalent to
\begin{equation}\label{phi:to:limit:1}
\lim_{j\to +\infty} \int_{\dR} \sum_{\vk \in [\dn_j^{-1}\dZ]}
\ff(\xi)\ol{\fg(\xi+2\pi \vk)} \I^{\dn_j \vk}_{\fPhi_j}(\dn_j\xi)d\xi=\la \ff, \fg\ra, \qquad \ff,\fg\in \dD.
\end{equation}
By our assumption in \eqref{cond:dn}, as in the proof of Theorem~\ref{thm:fndwf:criterion}, there exists a positive integer $J''$ such that $\ff(\xi)\ol{\fg(\xi+2\pi \vk)} I^{\dn_j\vk}_{\fPhi_j}(\dn_j\xi)=0$ for all $\xi\in \dR$, $\vk\in \gL\bs \{0\}$ and $j\ge J''$. Consequently, for $j\ge J''$, \eqref{phi:to:limit:1} becomes
\[
\lim_{j\to +\infty} \int_{\dR}
\ff(\xi)\ol{\fg(\xi)} \I^{\0}_{\fPhi_j}(\dn_j\xi)d\xi=\int_{\dR} \ff(\xi)\ol{\fg(\xi)}\, d\xi, \qquad \ff,\fg\in \dD,
\]
which is equivalent to \eqref{phi:to1}.
\end{proof}

We point out that \eqref{cond:mra} can be equivalently rewritten as
\begin{equation}\label{cond:mra:2}
\I_{\fPhi_j}^{\vk}(\xi)+
\I_{\fPsi_j}^{\vk}(\xi)=\I_{\fPhi_{j+1}}^{\dn_{j+1} \dn_j^{-1} \vk}(\dn_{j+1}\dn_j^{-1} \xi), \qquad a.e.\; \xi\in \dR, \vk\in \dZ\cup [(\dn_{j+1}\dn_j^{-1})^{-1}\dZ].
\end{equation}
Due to its similarity to MRA (multiresolution analysis) structure (\cite{Daub:book,Mallat:book,Meyer:book}),
the relations in \eqref{mra} (or more general, \eqref{quasirefinable}) could be called the quasi-MRA structure. One important property of dual wavelet frames is its frame approximation order (see \cite{DHRS}). Due to \eqref{quasirefinable}, we observe that the frame approximation order associated with the pair in Theorem~\ref{thm:fnws:special} is completely determined by the generators $\fPhi_j, \tilde \fPhi_j, j\ge J_0$ and it has nothing to do with $\fPsi_j, \tilde \fPsi_j, j\ge J_0$.

The following simple result on matrices will be needed later.

\begin{lemma}\label{lem:matrix}
Let $\dm$ be a $d\times d$ real-valued matrix. Assume that $\gl_1, \ldots, \gl_d$ are all the eigenvalues of $\dm$ ordered in such a way that $|\gl_1|\le |\gl_2|\le \cdots \le |\gl_d|$. Define $\gl_{\min}:=|\gl_1|$ and $\gl_{\max}:=|\gl_d|$.
Then for any $\gep>0$, there is a norm $\|\cdot \|_{\dm}$ on $\dR$ such that
\begin{equation}\label{matrix:norm}
(\gl_{\min}-\gep)\|x\|_{\dm} \le \|\dm x\|_{\dm} \le (\gl_{\max}+\gep)\|x\|_{\dm}, \qquad \forall\; x\in \dR.
\end{equation}
Moreover, if every eigenvalue $\gl$ of $\dm$ satisfying $|\gl|=\gl_{\min}$ has the same algebraic and geometric multiplicity, then $\gl_{\min}-\gep$ on the left-hand side of \eqref{matrix:norm} can be replaced by $\gl_{\min}$. Similarly, if every eigenvalue $\gl$ of $\dm$ satisfying $|\gl|=\gl_{\max}$ has the same algebraic and geometric multiplicity, then $\gl_{\max}-\gep$ on the right-hand side of \eqref{matrix:norm} can be replaced by $\gl_{\max}$. In particular, if $\dm$ is isotropic (that is, $\dm$ is similar to a diagonal matrix with all the diagonal entries having the same modulus), then $\|\dm x\|_{\dm}=|\det \dm|^{1/d} \|x\|_{\dm}$ for all $x\in \dR$. (Conversely, this identity implies that $\dm$ is isotropic.)
\end{lemma}

\begin{proof} There exists a $d\times d$ complex-valued invertible matrix $E$ such that $E \dm E^{-1}$ is the Jordan canonical form of $\dm$. Denote
\[
F=\left[\begin{matrix}
0 &1 &\cdots &0\\
\vdots &\vdots &\ddots &\vdots\\
0 &0 &\cdots &1\\
0 &0 &\cdots &0\end{matrix}\right].
\]
Let $\gl I_r+F$ be a Jordan block in the Jordan canonical form $E \dm E^{-1}$ of $\dm$. For any $\gep \ne 0$, we have $\mbox{diag}(1, \gep^{-1}, \gep^{-2}, \ldots, \gep^{1-r}) (\gl I_r +F) \mbox{diag}(1, \gep, \gep^{2}, \ldots, \gep^{r-1})=\gl I_r +\gep F$. Therefore, $\gl I_r+F$ is similar to $\gl I_r+\gep F$. Consequently, there exists a $d\times d$ complex-valued matrix $G$ such that $G\dm G^{-1}$ is a diagonal block matrix with each block being the form of $\gl I+\gep F$. Define a norm $\|\cdot\|_{\dm}$ on $\dC$ by $\|x\|_{\dm}^2:=\|Gx\|^2:=\bar{x}^T \ol{G}^T G x$. Restricting the norm $\|\cdot \|_{\dm}$ on $\dR$, now we can easily deduce that \eqref{matrix:norm} holds.
\end{proof}

As a direct consequence of Theorem~\ref{thm:fnws:special}, we have the following result on a pair of frequency-based nonhomogeneous dual wavelet frames in the distribution space.

\begin{cor}\label{cor:ndwf} Let $\dm$ be a $d\times d$ real-valued expansive matrix and define $\dn:=(\dm^T)^{-1}$. Let $\fPhi, \tilde \fPhi$ in \eqref{fPhi} and
\begin{equation}\label{fPsi}
\fPsi=\{\fpsi^1, \ldots, \fpsi^\mpsi\} \quad \mbox{and}\quad \tilde \fPsi=\{\tilde \fpsi^1, \ldots, \tilde \fpsi^\mpsi\}
\end{equation}
be finite subsets of $\dlLp{2}$. Then the following statements are equivalent:
\begin{enumerate}
\item[{\rm(i)}] $(\FWS_J(\fPhi; \fPsi), \FWS_J(\tilde \fPhi; \tilde \fPsi))$ is a pair of frequency-based nonhomogeneous dual $\dn$-wavelet frames in the distribution space $\dDpr$ for some integer $J$;
\item[{\rm(ii)}] item (i) holds for all integers $J$;

\item[{\rm(iii)}] for all $\ff, \fg\in \dD$, the following identities hold:
\begin{equation}\label{nws:to:0}
\lim_{j\to +\infty} \sum_{\ell=1}^{\mphi} \sum_{\vk\in \dZ} \la \ff, \fphi^\ell_{\dn^j; \0,\vk}\ra\la \tilde \fphi^\ell_{\dn^j;\0,\vk}, \fg\ra=(2\pi)^d \la \ff, \fg\ra
\end{equation}
and
\begin{equation}\label{nws:mra}
\begin{split}
\sum_{\ell=1}^\mphi \sum_{\vk\in \dZ} \la \ff, \fphi^\ell_{I_d; \0, \vk}\ra \la \tilde \fphi^\ell_{I_d;\0, \vk}, \fg\ra&+\sum_{\ell=1}^{\mpsi} \sum_{\vk\in \dZ} \la \ff, \fpsi^\ell_{I_d; \0, \vk}, \ra \la \tilde \fpsi^\ell_{I_d;\0,\vk}, \fg\ra\\
&=\sum_{\ell=1}^\mphi \sum_{\vk\in \dZ} \la \ff, \fphi^\ell_{\dn; \0, \vk}\ra \la \tilde \fphi^\ell_{\dn;\0, \vk}, \fg\ra;
\end{split}
\end{equation}

\item[{\rm(iv)}] $\lim_{j\to +\infty} \sum_{\ell=1}^\mphi \ol{\fphi^\ell(\dn^j\cdot)} \tilde \fphi^\ell(\dn^j\cdot)=1$ in the sense of distributions and
\begin{equation}\label{nws:I}
\I^\vk_\fPhi(\xi)+\I^\vk_\fPsi (\xi)=\I^{\dn \vk}_{\fPhi}(\dn \xi), \qquad a.e.\, \xi\in \dR, \, \vk\in \dZ \cup [\dn^{-1}\dZ],
\end{equation}
where $\I^\vk_{\fPhi}$ is defined in \eqref{I:Phi} and
\begin{equation}\label{I:fPsi}
\I^\vk_\fPsi(\xi):=\sum_{\ell=1}^\mpsi \ol{\fpsi^\ell(\xi)} \tilde \fpsi^\ell(\xi+2\pi \vk), \quad \vk\in \dZ\quad \mbox{and}\quad  \I^\vk_\fPsi:= 0, \quad \vk \in \dR\bs \dZ.
\end{equation}
%
%
%
\end{enumerate}
\end{cor}

\begin{proof} We first show that the condition in \eqref{cond:dn} is satisfied. Since $\dn_j=\dn^j$ and all the eigenvalues of $\dn$ are less than one in modulus, by Lemma~\ref{lem:matrix}, it is trivial to conclude that $\lim_{j\to \infty} \| \dn^j x\|=0$ for all $x\in \dR$. By Lemma~\ref{lem:matrix} again, it is also trivial to see that the set $\gL:=\cup_{j=J}^\infty [\dn^{-j} \dZ]=
\cup_{j=J}^\infty [(\dm^T)^{j} \dZ]$ is discrete. So, the condition in \eqref{cond:dn} is satisfied.

Now all the claims in Corollary~\ref{cor:ndwf} follow directly from Theorem~\ref{thm:fnws:special}.
\end{proof}

As another application of the notion of a pair of frequency-based nonstationary dual wavelet frames in the distribution space, we have the following result which naturally connects a nonstationary filter bank with wavelets and framelets in the distribution space.

\begin{theorem}\label{thm:fb}
Let $\{\dm_j\}_{j=1}^\infty$ be a sequence of $d\times d$ integer invertible matrices.
Define $\dn_0:=I_d$ and $\dn_j:=(\dm_1^T\cdots \dm_j^T)^{-1}$ for all $j \in \N$. Assume that \eqref{cond:dn} holds.
Let $\fa_j$ and $\tilde \fa_j$, $j\in \N$ be $2\pi\dZ$-periodic measurable functions
such that there exist positive numbers $\gep, \tilde \gep, C_j, \tilde C_j, j\in \N$ satisfying
\begin{equation}\label{mask:cond}
|1-\fa_j(\dn_j\xi)|\le C_j \|\xi\|^\gep, \qquad |1-\tilde \fa_j(\dn_j\xi)|\le \tilde C_j \|\xi\|^{\tilde \gep}, \qquad a.e.\; \xi\in \dR, j\in \N
\end{equation}
and
\begin{equation}\label{cond:C}
\sum_{j=1}^\infty C_j<\infty \quad \mbox{and}\quad
\sum_{j=1}^\infty \tilde C_j<\infty.
\end{equation}
For all $j\in \N \cup \{0\}$, define
\begin{equation}\label{Phi:ndwf}
\fphi_{j}(\xi):=\prod_{n=1}^\infty \fa_{j+n}(\dn_{j+n} \dn_j^{-1} \xi) \quad \hbox{and}\quad
\tilde \fphi_{j}(\xi):=\prod_{n=1}^\infty \tilde \fa_{j+n}(\dn_{j+n} \dn_j^{-1} \xi),
\qquad \xi\in \dR.
\end{equation}
Then all $\fphi_{j}$, $\tilde \fphi_{j}, j\in \N\cup\{0\}$ are well-defined functions in $\dlLp{\infty}$ satisfying
\begin{equation}\label{ns:refeq}
\fphi_{j-1}(\dm_j^T \xi)=\fa_j(\xi)\fphi_j(\xi)\quad \mbox{and} \quad \tilde \fphi_{j-1}(\dm_j^T\xi)=\tilde \fa_j(\xi)\tilde\fphi_j(\xi), \qquad
a.e.\, \xi\in \dR, j\in \N.
\end{equation}
Let $\fth_{j,1}, \ldots, \fth_{j,\mphi_{j-1}}$, $\fb_{j,1}, \ldots, \fb_{j,\mpsi_{j-1}}$ and
$\tilde \fth_{j,1}, \ldots, \tilde \fth_{j,\mphi_{j-1}}$, $\tilde \fb_{j,1}, \ldots, \tilde \fb_{j,\mpsi_{j-1}}$,
with $\mphi_{j-1}, \mpsi_{j-1}\in \N\cup\{0\}$ and $j\in \N$, be $2\pi\dZ$-periodic measurable functions in $\dlLp{2}$. For $j\in \N$, define
\begin{align}
&\fphi^{j-1,\ell}(\xi):=\fth_{j,\ell}(\xi)\fphi_{j-1}(\xi) \quad \mbox{and}\quad \tilde \fphi^{j-1,\ell}(\xi):=\tilde \fth_{j,\ell}(\xi)\tilde \fphi_{j-1}(\xi),\quad \ell=1, \ldots, \mphi_{j-1}, \label{Phi:ns}\\
&\fpsi^{j-1, \ell}(\dm_j^T \xi):=\fb_{j,\ell}(\xi)\fphi_j(\xi) \quad \hbox{and}\quad \tilde \fpsi^{j-1, \ell}(\dm_j^T \xi):=\tilde \fb_{j,\ell}(\xi)\tilde \fphi_j(\xi), \quad \ell=1, \ldots, \mpsi_{j-1}.\label{Psi:ns}
\end{align}
Then $\fPhi_j, \tilde \fPhi_j$ in \eqref{fPhij} and $\fPsi_j, \tilde \fPsi_j$ in \eqref{fPsij} are subsets of $\lLp{2}$ for all $j\in \N\cup\{0\}$. Moreover, the pair
\begin{equation}\label{mra:ndwf:seq}
(\FWS_J(\fPhi_J; \{ \fPsi_{j}\}_{j=J}^{\infty}), \FWS_J(\tilde \fPhi_J; \{\tilde \fPsi_{j}\}_{j=J}^{\infty}))
\end{equation}
is a pair of frequency-based nonstationary dual wavelet frames in the distribution space $\Dpr$ for every integer $J\ge 0$, if and only if, for all $j\in \N$,
\begin{align}
&\Theta_j(\dm_j^T\xi) \ol{\fa_j(\xi)}\tilde \fa_j(\xi)+\sum_{\ell=1}^{\mpsi_{j-1}} \ol{\fb_{j,\ell}(\xi)}\tilde \fb_{j,\ell}(\xi)=\Theta_{j+1}(\xi),\qquad a.e.\, \xi\in \sigma_{\fphi_j}\cap \sigma_{\tilde \fphi_j},\label{nep:1}\\
&\Theta_j(\dm_j^T\xi) \ol{\fa_j(\xi)}\tilde \fa_j(\xi+2\pi\omega)+\sum_{\ell=1}^{\mpsi_{j-1}} \ol{\fb_{j,\ell}(\xi)}\tilde \fb_{j,\ell}(\xi+2\pi \omega)=0,
\qquad a.e.\, \xi\in \sigma_{\fphi_j}\cap (\sigma_{\tilde \fphi_j}-2\pi \omega)\label{nep:2}
\end{align}
for all $\omega\in \dmfcg{\dm_j}\bs \{\0\}$ with $\dmfcg{\dm_j}:=[(\dm_j^T)^{-1}\dZ]\cap [0,1)^d$,
and
\begin{equation}\label{Theta:to1}
\lim_{j\to +\infty} \Theta_{j+1}(\dn_j\cdot)=1\qquad \mbox{in the sense of distributions},
\end{equation}
where
\begin{equation}\label{sigmaphij}
\sigma_{\fphi_j}:=\Big \{ \xi\in \dR\; : \; \sum_{\vk\in \dZ}
|\fphi_j(\xi+2\pi \vk)|\ne 0 \Big \}, \quad
\sigma_{\tilde \fphi_j}:=\Big \{ \xi\in \dR\; : \; \sum_{\vk\in \dZ}
|\tilde \fphi_j(\xi+2\pi \vk)| \ne 0 \Big\}
\end{equation}
and
\begin{equation}\label{def:Theta}
\Theta_{j}(\xi):=\sum_{\ell=1}^{\mphi_{j-1}} \ol{\fth_{j,\ell}(\xi)}\tilde \fth_{j,\ell}(\xi), \qquad j\in \N.
\end{equation}
\end{theorem}

\begin{proof} Using the inequality $|z|\le e^{|1-z|}$ for all $z\in \C$, by \eqref{mask:cond}, we deduce that
\begin{equation}\label{est:1}
\Big| \prod_{n=m_1}^{m_2} \fa_{j+n} (\dn_{j+n}\xi)\Big|
\le e^{\sum_{n=m_1}^{m_2} |1-\fa_{j+n}(\dn_{j+n}\xi)|}
\le e^{\|\xi\|^{\gep} \sum_{n=m_1}^{m_2} C_{j+n}}\le e^{C\|\xi\|^\gep}
\end{equation}
for all $m_1, m_2\in \N$ with $m_1<m_2$, where $C:=\sum_{j=1}^\infty C_j<\infty$.
On the other hand, we have the following identity
\[
1-\prod_{n=m_1}^{m_2} \fa_{j+n}(\dn_{j+n}\xi)=\sum_{m=m_1}^{m_2}
\Big( 1-\fa_{j+m}(\dn_{j+m}\xi)\Big)\Big( \prod_{n=m+1}^{m_2} \fa_{j+n}(\dn_{j+n}\xi)\Big),
\]
where we used the convention $\prod_{n=m_2+1}^{m_2}:=1$.
Therefore, by \eqref{mask:cond} and \eqref{est:1}, we have
\begin{equation}\label{est:2}
\Big| 1-\prod_{n=m_1}^{m_2} \fa_{j+n}(\dn_{j+n}\xi)\Big|
\le e^{C\|\xi\|^\gep} \sum_{m=m_1}^{m_2}
|1-\fa_{j+m}(\dn_{j+m}\xi)|\le e^{C\|\xi\|^\gep} \|\xi\|^\gep \sum_{m=m_1}^{m_2} C_{j+m}.
\end{equation}
Since $C=\sum_{j=1}^\infty C_j<\infty$, the above inequality implies the convergence in $\dlLp{\infty}$ of the infinite product $\prod_{n=1}^\infty \fa_{j+n}(\dn_{j+n}\xi)$.
Since $\fphi_{j}(\dn_j\xi)=\prod_{n=1}^\infty \fa_{j+n}(\dn_{j+n}\xi)$, it follows from \eqref{est:1} that $\fphi_j\in \dlLp{\infty}$. Since all $\fth_{j,\ell}, \fb_{j,\ell}\in \dlLp{2}$, it is evident that all $\fphi^{j,\ell}, \fpsi^{j,\ell}$ are elements in $\dlLp{2}$. Similarly, we can prove that $\tilde \fphi_j\in \dlLp{\infty}$ and all $\tilde \fphi^{j,\ell}, \tilde \fpsi^{j,\ell}$ are elements in $\dlLp{2}$.

By \eqref{Phi:ndwf}, we have $\fphi_{j-1}(\dn_{j-1}\xi)=\fa_j(\dn_j\xi)\fphi_j(\dn_j\xi)$, from which we can easily deduce that \eqref{ns:refeq} holds. Since $\dm_j^T \dn_j=\dn_{j-1}$ and $\dm_j$ is an integer matrix, we see that \eqref{cond:mra} with $j$ being replaced by $j-1$ is equivalent to
\begin{equation}\label{cond:mra:fb}
\begin{split}
\sum_{\ell=1}^{\mphi_{j-1}}\ol{\fphi^{j-1,\ell}(\dm_j^T\xi)}
\tilde \fphi^{j-1,\ell}(\dm_j^T(\xi+2\pi \vk+2\pi \omega))
&+\sum_{\ell=1}^{\mpsi_{j-1}} \ol{\fpsi^{j-1,\ell}(\dm_j^T\xi)}
\tilde \fpsi^{j-1,\ell}(\dm_j^T (\xi+2\pi \vk+2\pi \omega))\\
&=\gd(\omega) \sum_{\ell=1}^{\mphi_j} \ol{\fphi^{j,\ell}(\xi)}\tilde \fphi^{j,\ell}(\xi+2\pi \vk+2\pi \omega)
\end{split}
\end{equation}
for almost every $\xi\in \dR$ and for all $\vk\in \dZ$ and $\omega\in \dmfcg{\dm_j}$, where $\gd(0)=1$ and $\gd(\omega)=0$ for all $\omega \ne 0$. Now by \eqref{Phi:ns} and \eqref{Psi:ns}, it is easy to deduce that \eqref{cond:mra:fb} is equivalent to
\eqref{nep:1} and \eqref{nep:2}.

By \eqref{est:2}, we deduce that
\begin{equation}\label{est:3}
|1-\fphi_j(\dn_j\xi)|\le e^{C \|\xi\|^\gep}\|\xi\|^\gep\sum_{m=j+1}^\infty C_m
\quad\mbox{and}\quad
|1-\tilde\fphi_j(\dn_j\xi)|\le e^{\tilde C \|\xi\|^{\tilde \gep}}\|\xi\|^{\tilde \gep} \sum_{m=j+1}^\infty \tilde C_m,
\end{equation}
where $\tilde C:=\sum_{j=1}^\infty \tilde C_j<\infty$. By \eqref{cond:C}, the above inequalities imply that for any bounded set $K$,
there exists a positive integer $J_K$ such that
\begin{equation}\label{est:4}
\frac{1}{2}\le |\fphi_j(\dn_j\xi)|\le \frac{3}{2}, \qquad
\frac{1}{2}\le |\tilde \fphi_j(\dn_j\xi)|\le \frac{3}{2}, \qquad a.e.\; \xi\in K, j\ge J_K.
\end{equation}
It follows from \eqref{Phi:ns} that
\[
\sum_{\ell=1}^{\mphi_j} \ol{\fphi^{j,\ell}(\dn_j\xi)}\tilde \fphi^{j,\ell}(\dn_j\xi)=\Theta_{j+1}(\dn_j\xi)
\ol{\fphi_j(\dn_j\xi)} \tilde \fphi_j(\dn_j\xi).
\]
By \eqref{est:3}, \eqref{est:4} and the above identity, using Lebesgue dominated convergence theorem, we conclude that $\lim_{j\to \infty} \sum_{\ell=1}^{\mphi_j} \ol{\fphi^{j,\ell}(\dn_j\cdot)}\tilde \fphi^{j,\ell}(\dn_j\cdot)=1$ in the sense of distributions if and only if \eqref{Theta:to1} holds. The proof is now completed by Theorem~\ref{thm:fnws:special}.
\end{proof}

The conditions in \eqref{nep:1} and \eqref{nep:2} in Theorem~\ref{thm:fb} generalize the well-known Oblique Extension Principle (OEP) in \cite{CHS,DH:dwf,DHRS} (also see \cite{E,EH,Han:oep,HanMo:acha}) from homogeneous wavelet systems to nonhomogeneous and nonstationary wavelet systems. The stationary case of OEP with $\mphi_j=1$ and $\fth_{j,1}=\tilde \fth_{j,1}=1$ is given in \cite{RonShen:dwf,RonShen:twf} for homogeneous wavelet systems.
Let $\dm$ be a $d\times d$ integer expansive matrix and $\dn:=(\dm^T)^{-1}$. In Theorem~\ref{thm:fb}, if $\dm_j=\dm, \fa_j=\fa, \tilde \fa_j=\tilde \fa$ for all $j\in \N$ and if there exist positive numbers $\gep, \tilde \gep, C, \tilde C$ such that
\begin{equation}\label{mask:cond:a:ta}
|1-\fa(\xi)|\le C\|\xi\|^\gep \quad \mbox{and}\quad
|1-\tilde\fa(\xi)|\le C\|\xi\|^{\tilde \gep}, \qquad a.e.\, \xi\in \dR,
\end{equation}
by Lemma~\ref{lem:matrix}, then it is easy to see that \eqref{cond:dn},
\eqref{mask:cond}, and \eqref{cond:C} are satisfied.

\section{Nonhomogeneous and Directional Nonstationary Tight Framelets in $\dLp{2}$}

In this section, we shall study and construct nonhomogeneous (stationary) tight wavelet frames with a minimum number of generators by proving Theorem~\ref{thm:ntwf:special}.
By modifying the construction in Theorem~\ref{thm:ntwf:special}, we also show that we can easily construct directional nonstationary (more precisely, nonhomogeneous quasi-stationary) tight wavelet frames in $\dLp{2}$. Moreover, both constructed systems are associated
with filter banks and can be derived by Theorem~\ref{thm:fb} with such filter banks if the underlying dilation matrix is an integer expansive matrix.

As a direct consequence of Corollary~\ref{cor:fntwf} and Theorem~\ref{thm:fnws:special}, we have the following result.

\begin{cor}\label{cor:ntwf:special}
Let $\dm$ be a $d\times d$ real-valued expansive matrix and define $\dn:=(\dm^T)^{-1}$. Let $J_0$ be an integer. Let $\fPhi$ in \eqref{fPhi} and $\fPsi_j$ in \eqref{fPsij} be finite subsets of $\dlLp{2}$ for all $j\ge J_0$.
Then $\FWS_J(\fPhi; \{\fPsi_j\}_{j=J}^\infty)$ is a frequency-based nonstationary tight wavelet frame in $\dLp{2}$ for every integer $J\ge J_0$, that is, $\fPhi, \fPsi_j \subseteq \dLp{2}$ for all $j\ge J_0$ and
\begin{equation}\label{def:ntwf}
\sum_{\ell=1}^\mphi \sum_{\vk\in \dZ} |\la \ff, \fphi^\ell_{\dn^J; \0, \vk}\ra|^2+\sum_{j=J}^\infty \sum_{\ell=1}^{\mpsi_j} |\la \ff, \fpsi^{j,\ell}_{\dn^j; \0, \vk}\ra|^2=(2\pi)^d \|\ff\|_{\dLp{2}}^2, \qquad \forall\; \ff \in \dLp{2}, J\ge J_0,
\end{equation}
if and only if,
\begin{equation}\label{phi:to:1:special}
\lim_{j\to +\infty} \sum_{\ell=1}^\mphi |\fphi^\ell(\dn^j\cdot)|^2=1\quad \mbox{in the sense of distributions}
\end{equation}
and for all $j\ge J_0$,
\begin{equation}\label{cond:1}
\begin{split}
\sum_{\ell=1}^\mphi \ol{\fphi^\ell(\xi)}\fphi^\ell(\xi+2\pi \vk)&+
\sum_{\ell=1}^{\mpsi_j} \ol{\fpsi^{j,\ell}(\xi)}\fpsi^{j,\ell}(\xi+2\pi \vk)\\
&=\sum_{\ell=1}^\mphi \ol{\fphi^\ell(\dn \xi)}\fphi^\ell(\dn(\xi+2\pi \vk)), \qquad a.e.\, \xi\in \dR, \vk\in \dZ \cap [\dn^{-1}\dZ],
\end{split}
\end{equation}
\begin{align}
&\sum_{\ell=1}^\mphi \ol{\fphi^\ell(\xi)}\fphi^\ell(\xi+2\pi \vk)+
\sum_{\ell=1}^{\mpsi_j} \ol{\fpsi^{j,\ell}(\xi)}\fpsi^{j,\ell}(\xi+2\pi \vk)=0, \qquad a.e.\, \xi\in \dR, \vk\in \dZ \bs [\dn^{-1}\dZ],\label{cond:2}\\
&\sum_{\ell=1}^\mphi \ol{\fphi^\ell(\dn \xi)}\fphi^\ell(\dn(\xi+2\pi \vk))=0, \qquad a.e.\, \xi\in \dR, \vk\in [\dn^{-1}\dZ]\bs \dZ.\label{cond:3}
\end{align}
Moreover, if the following property holds:
\begin{equation}\label{special:gen}
\fh(\xi)\fh(\xi+2\pi \vk)=0 \qquad a.e.\; \xi\in \dR, \vk \in \dZ\bs \{\0\}, \fh\in \fPhi\cup (\cup_{j=J_0}^\infty \fPsi_j),
\end{equation}
then all the conditions in \eqref{cond:1}, \eqref{cond:2}, and \eqref{cond:3} are reduced to the following simple condition
\begin{equation}\label{nep:special}
\sum_{\ell=1}^\mphi |\fphi^\ell(\xi)|^2+\sum_{\ell=1}^{\mpsi_j} |\fpsi^{j,\ell}(\xi)|^2=\sum_{\ell=1}^\mphi |\fphi^\ell(\dn\xi)|^2, \qquad a.e.\, \xi\in \dR, j\ge J_0.
\end{equation}
\end{cor}

\begin{proof}
By the proof of Corollary~\ref{cor:ndwf}, we see that the condition in \eqref{cond:dn} is satisfied.  It is trivial to check that
\eqref{cond:mra} with $\fPhi_j:=\tilde \fPhi_j:=\fPhi$ and $\tilde \fPsi_j:=\fPsi_j$ is equivalent to all the conditions in \eqref{cond:1}, \eqref{cond:2}, and \eqref{cond:3}.
Note that the assumption in \eqref{special:gen} simply means $\I^\vk_{\fPhi}(\xi)=0$ and $\I^{\vk}_{\fPsi_j}(\xi)=0$ for almost every $\xi\in \dR$, $\vk\in \dR \bs \{\0\}$, and $j\ge J_0$.
Now it is easy to see that all the conditions in \eqref{cond:1}, \eqref{cond:2}, and \eqref{cond:3} are reduced to the simple condition
in \eqref{nep:special}.
Since $\fPhi_j=\tilde \fPhi_j=\fPhi$ and $\mphi_j=\mphi$, we have $\sum_{\ell=1}^{\mphi_j} \ol{\fphi^{j,\ell}(\dn_j\xi)}\tilde \fphi^{j,\ell}(\dn_j\xi)=
\sum_{\ell=1}^\mphi |\fphi^\ell(\dn^{j}\xi)|^2$. Now it is straightforward to see that \eqref{phi:to1} is equivalent to \eqref{phi:to:1:special}. The proof is completed by Theorem~\ref{thm:fnws:special} and Corollary~\ref{cor:fntwf}.
\end{proof}

The sequence of systems $\FWS_J(\fPhi; \{\fPsi_j\}_{j=J}^\infty), J\ge J_0$ in Corollary~\ref{cor:ntwf:special} is of particular interest in applications. The only difference between these systems in Corollary~\ref{cor:ntwf:special} and nonhomogeneous wavelet systems lies in that the sets $\fPsi_j$ of wavelet generators depend on the scale level $j$ and therefore, the wavelet generators are not stationary.
This freedom of using different sets $\fPsi_j$ at different scales allows us to have different ways of time-frequency partitions at different resolution scales and therefore such a freedom is of importance in certain applications. Since these systems in Corollary~\ref{cor:ntwf:special} are between nonhomogeneous stationary wavelet systems and (fully) nonstationary wavelet systems, we shall call them nonhomogeneous quasi-stationary wavelet systems.

Now we are ready to prove Theorem~\ref{thm:ntwf:special} in Section~1.

\begin{proof}[Proof of Theorem~\ref{thm:ntwf:special}]
Let $\gl_{\min}$ be the absolute value of the smallest eigenvalue of $\dm$ in modulus. Note that $\gl_{\min}>1$. Define $B_{c}(\0):=\{ \xi\in \dR \, : \, \|\xi\|_{\dm^T}<c\}$, where $\|\cdot \|_{\dm^T}$ is the norm in Lemma~\ref{lem:matrix} with $\dm$ being replaced by $\dm^T$ and with an $\gep>0$ satisfying $\gl:=\gl_{\min}-\gep>1$. By  \eqref{matrix:norm}, we see that $B_c(\0)\subseteq B_{\gl c}(\0)\subseteq \dm^T B_c(\0)$.
Choose $\gl_0$ such that $0<\gl_0<1$. For small enough $c$, we have
\begin{equation}\label{Bc0}
\quad B_c(\0)\subseteq B_{\gl c}(\0)\subseteq \dm^T B_c(\0) \subseteq (-\gl_0\pi, \gl_0\pi)^d \quad \mbox{and}\quad
(\dm^T)^2 B_c(\0)\subseteq (-\gl_0\pi, \gl_0\pi)^d.
\end{equation}
For any $\rho>0$, there is an even function $h_\rho$ in $C^\infty(\R)$ such that $h_\rho$ takes value one on $(-\rho/2,\rho/2)$, is positive on $(-\rho,\rho)$, and vanishes on $\R \bs (-\rho,\rho)$.
Such a function $h_\rho$ can be easily constructed, for example, see \cite[Section~4]{Han:frame}. For sufficiently small $\rho>0$, we define
\begin{equation}\label{fphi:sp:0}
\fphi:=\frac{\chi_{B_{(1+\gl)c/2}(\0)}*h_\rho(\|\cdot\|)}{
[\chi_{B_{(1+\gl)c/2}(\0)}*h_\rho(\|\cdot\|)](0)}.
\end{equation}
For sufficiently small $\rho>0$, it is evident that $\fphi$ is an even function in $C^\infty(\dR)$, $0\le \fphi(\xi)\le 1$ for all $\xi\in \dR$, and
\begin{equation}\label{fphi:sp:1}
\fphi(\xi)=1 \qquad \forall\; \xi\in B_c(\0) \quad \mbox{and}\quad \fphi(\xi)=0 \qquad \forall \; \xi\in \dR\bs (\dm^TB_c(\0)).
\end{equation}
%
Now we define
\begin{equation}\label{fpsi:sp:0}
\fpsi(\xi):=\sqrt{\feta(\xi)}\quad \mbox{with}\quad \feta(\xi):=(\fphi(\dn\xi))^2-(\fphi(\xi))^2, \qquad \xi\in \dR,
\end{equation}
where $\dn:=(\dm^T)^{-1}$. By \eqref{Bc0} and \eqref{fphi:sp:1}, we see that $\feta$ vanishes on $B_c(\0)$ and is supported inside $(\dm^T)^2 B_c(\0)\subseteq (-\gl_0\pi, \gl_0\pi)^d$.

Since $\fphi\in C^\infty(\dR)$, it is obvious that $\feta\in C^\infty(\dR)$. Since $0\le \fphi\le 1$ and by \eqref{fphi:sp:1},
\begin{equation}\label{fphi:sp:2}
\fphi(\dn\xi)=1 \quad \mbox{whenever}\; \fphi(\xi)\ne 0.
\end{equation}
We deduce from \eqref{fphi:sp:2} that $0\le \feta(\xi)\le (\fphi(\dn\xi))^2$. Therefore, $\fpsi$ is well defined and
\begin{equation}\label{fpsi:sp:1}
0\le \fpsi(\dm^T \xi)\le \fphi(\xi)\le 1, \qquad \xi\in \dR \quad \mbox{and}\quad \fpsi(\dm^T\xi)=\fphi(\xi), \qquad \xi\in \dR \bs B_c(\0).
\end{equation}
%
To prove that $\fpsi\in C^\infty(\dR)$, we first show that if $\fphi(\xi_0)=0$ or $1$, then all the derivatives of $\fphi$ at $\xi_0$ must vanish. In fact, if $\fphi(\xi_0)=0$, then by the definition of $\fphi$ in \eqref{fphi:sp:0} we see that the supports of $h_\rho(\|\cdot\|)$ and $\chi_{B_{(1+\gl)c/2}(\0)}(\xi_0-\cdot)$ are essentially disjoint. Consequently, it follows from \eqref{fphi:sp:0} that all the derivatives of $\fphi$ must vanish at $\xi_0$. Similarly, if $\fphi(\xi_0)=1$, then by \eqref{fphi:sp:0} we see that the support of $h_\rho(\|\cdot\|)$ is essentially contained inside the support of $\chi_{B_{(1+\gl)c/2}(\0)}(\xi_0-\cdot)$. Consequently, it follows from \eqref{fphi:sp:0} that all the derivatives of $\fphi$ must vanish at $\xi_0$.

Now we show that $\fpsi\in C^\infty(\dR)$. In fact, if $\feta(\xi_0)\ne 0$, since $\feta\in C^\infty$, then it is trivial to see that $\fpsi=\sqrt{\feta}$ is infinitely differentiable at $\xi_0$. If $\feta(\xi_0)=0$, by \eqref{fpsi:sp:0} and \eqref{fphi:sp:2}, we must have either $\fphi(\xi_0)=\fphi(\dn \xi_0)=0$ or $\fphi(\xi_0)=\fphi(\dn \xi_0)=1$. By what has been proved, we see that all the derivatives of $\feta$ must vanish at $\xi_0$. Using the Taylor expansion of $\feta$ in a neighborhood of $\xi_0$, now it is easy to see that $\fpsi=\sqrt{\feta}$ must be infinitely differentiable at $\xi_0$ with all its derivatives at $\xi_0$ being zero.
Therefore, $\fpsi\in C^\infty(\dR)$.

Hence, we constructed two functions $\fphi, \fpsi\in C^\infty(\dR)$ vanishing outside $[-\pi,\pi]^d$ and
\begin{equation}\label{relation:special}
|\fphi(\xi)|^2+|\fpsi(\xi)|^2=|\fphi(\dn\xi)|^2 \qquad \forall\; \xi\in \dR.
\end{equation}
Define $\phi$ and $\psi$ via $\hat \phi:=\fphi$ and $\hat \psi:=\fpsi$. By Corollary~\ref{cor:ntwf:special}, all items (i)--(iii) of Theorem~\ref{thm:ntwf:special} hold.

Now we define $2\pi\dZ$-periodic functions $\fa_\vk, \vk\in \dZ$ by
\begin{equation}\label{fa}
\fa_\vk(\xi):=e^{-\iu \vk \cdot \dm^T\xi} \fphi(\dm^T\xi), \qquad \xi\in [-\pi,\pi)^d, \; \vk\in\dZ.
\end{equation}
By \eqref{fphi:sp:2}, $\fphi(\xi)=1$ whenever $\fphi(\dm^T\xi)\ne 0$.
Since $\fphi(\dm^T\cdot)$ vanishes outside $B_c(\0)\subseteq (-\gl_0\pi, \gl_0\pi)^d$,
now it is straightforward to see that $\fa_\vk\in C^\infty(\dT)$ and
$\fa_\vk(\xi)\fphi(\xi)=e^{-i\vk\cdot \dm^T\xi} \fphi(\dm^T\xi)$ for all $\xi\in \dR$.

Let $\fh_{\gl_0}$ be a $C^\infty$ function such that $\fh_{\gl_0}=1$ on $[-\gl_0\pi, \gl_0\pi]^d$ and $\fh_{\gl_0}=0$ on $[-\pi,-\pi)^d \bs [-(\gl_0+\gep_0)\pi, (\gl_0+\gep_0)\pi]^d$, where $0<\gep_0<1-\gl_0$.
Such a function $\fh_{\gl_0}$ can be easily constructed using tensor product (see \cite[Section~4]{Han:frame}). Now we define $2\pi\dZ$-periodic functions $\fb_\vk, \vk\in \dZ$ by
\begin{equation}\label{fb}
\fb_\vk(\xi):=\begin{cases}
e^{-\iu \vk \cdot \dm^T\xi} \fh_{\gl_0}(\xi) \frac{\fpsi(\dm^T\xi)}{\fphi(\xi)},\qquad  &\text{$\fphi(\xi)\ne 0$ and $\xi\in [-\pi,\pi)^d$;}\\
e^{-\iu \vk \cdot \dm^T\xi} \fh_{\gl_0}(\xi),\qquad  &\text{$\fphi(\xi)=0$ and $\xi\in [-\pi,\pi)^d$.}
\end{cases}
\end{equation}
By \eqref{fpsi:sp:1}, it is trivial to see that $e^{-i\vk \cdot \dm^T\xi} \fpsi(\dm^T\xi)=\fb_\vk(\xi)\fphi(\xi)$ for all $\xi\in \dR$. Now we show that $\fb_\vk\in C^\infty(\dT)$.
Note that $\fb_\vk$ is infinitely differentiable at $\xi$  for all $\xi\in [-\pi,\pi)^d$ such that $\fphi(\xi)\ne 0$.
By the fact that $\fh_{\gl_0}$ vanishes outside $[-(\gl_0+\gep_0)\pi, (\gl_0+\gep_0)\pi]^d$ and $\gl_0+\gep_0<1$, by \eqref{fpsi:sp:1}, we conclude that $\fb_\vk\in C^\infty(\dT)$.
Therefore, item (iv) holds.

Note that \eqref{twf:special} is a direct consequence of Proposition~\ref{prop:frame}. This completes the proof.
\end{proof}

When $\dm$ in Theorem~\ref{thm:ntwf:special} is a $d\times d$ integer expansive matrix, the nonhomogeneous tight wavelet frames in Theorem~\ref{thm:ntwf:special} can be obtained via Theorem~\ref{thm:fb} from its filer bank $(\fa_\0, \fb_\0)$ with
\[
\mphi_j=1, \quad \mpsi_j=1 \quad\mbox{and}\quad \fa_j=\tilde \fa_j=\fa_\0,\quad
\fb_{j,1}=\tilde \fb_{j,1}=\fb_\0,\qquad \fth_{j,1}=\tilde \fth_{j,1}=1, \qquad j\in \Z.
\]

One potential shortcoming of the construction in Theorem~\ref{thm:ntwf:special} is that the support of $\hat \psi$ has a ring structure and $\psi$ behaves like a radial basis function.
In fact, the functions $\hat \phi$ and $\hat \psi$ can be constructed to be radial basis functions for special dilation matrices such as $\dm=2I_d$, see \cite{Han:frame} for more details.
However, as pointed out in \cite{CD}, for high dimensional problems, directionality of a transform is an important feature. All stationary wavelets and framelets have the isotropic structure and cannot capture very well singularities other than the point singularity in high dimensions.

In the rest of this section, going from frequency-based nonhomogeneous tight wavelet frames to frequency-based nonhomogeneous quasi-stationary tight wavelet frames, we shall see that directionality can be easily achieved by modifying the construction in Theorem~\ref{thm:ntwf:special} and using a simple splitting technique in \cite{Han:frame}. The key lies in that $\fphi$ has a very small support and $\fpsi$ is linked to $\fphi$ via the relation in \eqref{relation:special}. By \eqref{nep:special} in Corollary~\ref{cor:ntwf:special},  the single generator $\fpsi$ in \eqref{relation:special} at every scale level can be easily split into many generators $\fpsi^{j,1}, \ldots, \fpsi^{j,\mpsi_j}$ with directionality so that
\begin{equation}\label{fphi:fpsij}
\sum_{\ell=1}^{\mpsi_j} |\fpsi^{j,\ell}(\xi)|^2=|\fpsi(\xi)|^2, \qquad a.e.\; \xi\in \dR.
\end{equation}

In the following, we describe the idea of the splitting technique by increasing angle resolution.
For each integer $j$, pick a positive integer $\mpsi_j$ and construct $C^\infty$ functions $\beta_{j,1}, \ldots, \beta_{j,\mpsi_j}$ on the unit sphere of $\dR$ such that
\begin{equation}\label{partition}
\sum_{\ell=1}^{\mpsi_j} |\beta_{j,\ell}(\xi)|^2=1\qquad \forall\; \xi\in \dR \quad \mbox{with}\quad \|\xi\|=1.
\end{equation}
To capture various types of singularities, preferably the significant energy part of each $\beta_{j,\ell}$ concentrates near a point or an $n$-dimensional manifold on the unit sphere with $0\le n<d$.
Let $\fphi$ and $\fpsi$ be constructed as in the proof of Theorem~\ref{thm:ntwf:special}.
Now, we can split $\fpsi$ by defining
\begin{equation}\label{wavelet:directional}
\fpsi^{j,\ell}(\xi):=\fpsi(\xi)\beta_{j,\ell}(\tfrac{\xi}{\|\xi\|}), \qquad \ell=1, \ldots, \mpsi_j\quad \mbox{and}\quad j\in \Z.
\end{equation}
Since $\fpsi$ vanishes in a neighborhood of the origin, all $\fpsi^{j,\ell}\in C^\infty(\dR)$ and \eqref{fphi:fpsij} is obviously true. Consequently, for all $j\in \Z$,
\begin{equation}\label{fphi:fpsij:sp}
|\fphi(\xi)|^2+\sum_{\ell=1}^{\mpsi_j}|\fpsi^{j,\ell}(\xi)|^2=|\fphi(\dn \xi)|^2, \qquad \xi\in \dR.
\end{equation}
Define $\fPsi_j:=\{\fpsi^{j,1}, \ldots, \fpsi^{j,\mpsi_j}\}$.
Now by Corollary~\ref{cor:ntwf:special} and Theorem~\ref{thm:ntwf:special}, $\FWS_J(\{\fphi\}; \{\fPsi_j\}_{j=J}^\infty)$ is a frequency-based nonstationary tight wavelet frame in $\dLp{2}$ for all integers $J$.

We mention that the main idea of the construction in Theorem~\ref{thm:ntwf:special} and the above splitting technique already appeared in \cite[Proposition~3.8 and Section~4]{Han:frame} for $\dm=2I_d$ and for homogeneous tight wavelet frames in $\dLp{2}$.

Many directional nonhomogeneous quasi-stationary (and nonstationary) tight wavelet frames in $\dLp{2}$ can be constructed. To illustrate the above general procedure, using polar coordinates, we present the following result for the special case $d=2$ and $\dm=2I_2$. In the following result and its proof, note that $\pcr$ and $\theta$ are used for polar coordinates and a complex number $\pcr e^{\iu \theta}$ is identified with the point $(\pcr\cos \theta, \pcr\sin \theta)$, that is, $\R^2$ is identified with the complex plane $\C$.

\begin{theorem}\label{thm:directional}
Let $m$ be a positive integer and $0\le \rho<1$. Then there exist two real-valued functions $\phi$ and $\feta$ in $\dLp{2}\cap C^\infty(\R^2)$ satisfying all the following properties:
\begin{enumerate}
\item[{\rm(i)}] $\hat \phi$ is a compactly supported  radial basis function in $C^\infty(\R^2)$ and there exists a $2\pi\Z^2$-periodic trigonometric function $\fa$ in $C^\infty(\T^2)$ such that $\hat \phi(2\xi)=\fa(\xi) \hat \phi(\xi)$ for all $\xi\in \R^2$;
\item[{\rm(ii)}] $\feta$ has the tensor-product structure in polar coordinates and there exist positive real numbers $\pcr_1, \pcr_2, \theta_0$ such that $\theta_0<\frac{\pi}{m}$ and
\begin{equation}\label{feta:supp}
\mbox{supp}\, \feta=\{ \pcr e^{\iu \theta}\; : \; \pcr_1\le \pcr\le \pcr_2, -\theta_0\le \theta \le \theta_0\};
\end{equation}
\item[{\rm(iii)}] for every nonnegative integer $J$, we have a nonstationary tight wavelet frame in $L_2(\R^2)$:
\begin{equation}\label{ntwf:2D}
\sum_{\vk\in \Z^2} |\la f, \phi_{2^J I_2;\vk}\ra|^2+\sum_{j=J}^\infty \sum_{\ell=1}^{\mpsi_j} \sum_{\vk\in \Z^2} |\la f, \psi^{j,\ell}_{2^jI_2; \vk}\ra|^2=\|f\|_{L_2(\R^2)}^2, \qquad \forall\, f\in L_2(\R^2),
\end{equation}
\end{enumerate}
where $s_j:=m 2^{\lfloor \rho j\rfloor}$, $\lfloor\cdot \rfloor$ is the floor function,
and all $\psi^{j,\ell}$ are real-valued functions in the Schwarz class satisfying the following properties:
\begin{enumerate}
\item[{\rm(1)}] $\psi^{j,0}$ is defined as follows: for $\pcr\ge 0$ and $\theta\in [-\pi,\pi)$,
\begin{equation}
\wh{\psi^{j,0}}(\pcr e^{i\theta}):=
\begin{cases} \feta(\pcr e^{i 2^{\lfloor \rho j\rfloor} \theta})+\feta(-\pcr e^{i 2^{\lfloor \rho j\rfloor} \theta}), &\text{if $\theta\in [-2^{-\lfloor \rho j\rfloor}\pi, 2^{-\lfloor \rho j\rfloor}\pi)$,}\\
0, &\text{if $\theta\in [-\pi,\pi)\bs [-2^{-\lfloor \rho j\rfloor}\pi, 2^{-\lfloor \rho j\rfloor}\pi)$;}
\end{cases}
\end{equation}
\item[{\rm(2)}] all other generators $\psi^{j,\ell}$ are obtained via rotations from $\psi^{j,0}$:  for all $\ell=1, \ldots, \mpsi_j$,
\begin{equation}\label{psi:rotation}
\psi^{j,\ell}(\pcr e^{i\theta}):=\psi^{j,0}( \pcr e^{i\theta} e^{i 2^{-\lfloor \rho j\rfloor}\pi (\ell-1)/m}), \qquad\pcr \ge 0, \theta \in [-\pi, \pi);
\end{equation}
\item[{\rm(3)}] the support of $\wh{\psi^{j,0}_{2^j I_2; \vk}}$ has two parts which are symmetric about the origin and each part obeys $\mbox{width}\approx \mbox{length}^{1-\rho}$. More precisely, for all $\vk\in \Z^2$ and $j\ge 0$,
\begin{equation} \label{psij:supp}
\mbox{supp}\, \wh{\psi^{j,0}_{2^j I_2; \vk}}=
\{ \pcr e^{\iu \theta}, -\pcr e^{\iu \theta} \; : \; 2^j \pcr_1\le \pcr\le 2^j \pcr_2, -2^{j-\lfloor \rho j\rfloor} \theta_0\le \theta \le 2^{j-\lfloor \rho j\rfloor} \theta_0\};
\end{equation}
\item[{\rm(4)}] $\wh{\psi^{j,\ell}}$ are compactly supported functions in $C^\infty(\R^2)$ vanishing in a neighborhood of the origin and there exist $2\pi\Z^2$-periodic trigonometric functions $\fb_{j,\ell}$ in $C^\infty(\T^2)$ such that
%
\begin{equation}\label{psij:refeq}
\wh{\psi^{j,\ell}}(2\xi)=\fb_{j,\ell}(\xi)\hat \phi(\xi), \qquad \xi\in \R^2, \; \ell=1, \ldots, \mpsi_j, \, j\in \N \cup\{0\}.
\end{equation}
\end{enumerate}
\end{theorem}

\begin{proof} To prove Theorem~\ref{thm:directional}, we construct some special $2\pi$-periodic functions.
For a positive integer $n$ and any $0<\gep\le \frac{\pi}{2n}$,
by \cite[Lemma 4.2 or (4.2)]{Han:frame}, we can construct a $2\pi$-periodic nonnegative function $\alpha_{n,\gep}$ in $C^\infty(\T)$ such that
\begin{equation}\label{alpha:small}
\alpha_{n,\gep}(\xi)=1, \qquad \xi\in [-\tfrac{\pi}{2n}+\gep, \tfrac{\pi}{2n}-\gep]\quad \hbox{and}\quad
\alpha_{n,\gep}(\xi)=0, \qquad \xi\in [-\pi,\pi)\bs [-\tfrac{\pi}{2n}-\gep, \tfrac{\pi}{2n}+\gep]
\end{equation}
and
\begin{equation}\label{alpha:tile}
\sum_{\ell=0}^{2n-1} |\alpha_{n,\gep}(\xi+\tfrac{\pi \ell}{n})|^2=1, \qquad \xi\in \R.
\end{equation}
%

Let $\mathring{\fphi}$ and $\mathring{\fpsi}$ be constructed in the proof of Theorem~\ref{thm:ntwf:special} with $d=1$ and $\dm=2$. Or equivalently, by \cite[Lemma 4.2 or (4.2)]{Han:frame}, we can easily construct two real-valued even functions $\mathring{\fphi}$ and $\mathring{\fpsi}$ in $C^\infty(\R)$ such that their supports are contained inside $[-\pi,\pi]$, $\fphi$ takes value one in a neighborhood of $0$, \eqref{relation:special} holds with $\dn=1/2$ and $d=1$, and
there exist $2\pi$-periodic trigonometric functions $\mathring{\fa}$ and $\mathring{\fb}$ in $C^\infty(\T)$ such that
%
\begin{equation}\label{1d:refeq}
\mathring{\fphi}(2\xi)=\mathring{\fa}(\xi)\mathring{\fphi}(\xi) \quad \mbox{and}\quad \mathring{\fpsi}(2\xi)=\mathring{\fb}(\xi)\mathring{\fphi}(\xi), \qquad \xi\in \R.
\end{equation}
Now we define
\begin{equation}\label{2D:fphi:fpsi}
\fphi(\pcr e^{i\theta})=\mathring{\fphi}(\pcr),\quad  \fpsi(\pcr e^{i\theta})=\mathring{\fpsi}(\pcr) \qquad
\mbox{and}\quad \feta(\pcr e^{i\theta}):=\mathring{\fpsi}(\pcr) \alpha_{m,\gep}(\theta), \qquad \pcr\ge 0, \theta\in [-\pi, \pi).
\end{equation}
One can easily check that the supports of all $\wh{\psi^{j,\ell}}$ are contained inside $(-\pi,\pi)^d$ and \eqref{fphi:fpsij:sp} holds with $d=2$ and $\dn=\frac{1}{2}I_2$ for $j\in \N \cup\{0\}$. Define $2\pi\Z^2$-periodic functions $\fa, \fb_{j,0}$ on $[-\pi,\pi)^2$ by
\begin{equation}\label{direction:fa}
\fa(\pcr e^{i\theta})=\begin{cases}
\mathring{\fa}(\pcr), &\text{if $0\le \pcr\le \pi, \theta\in [-\pi, \pi)$;}\\
0,&\text{if $\pcr e^{i\theta}\in [-\pi,\pi)^2$ with $\pcr>\pi$}\end{cases}
\end{equation}
and
\[
\fb_{j,0}(\pcr e^{i\theta})=\begin{cases}
\mathring{\fb}(\pcr)\Big(\alpha_{m,\gep}(2^{\lfloor \rho j\rfloor}\theta)+
\alpha_{m,\gep}(2^{\lfloor \rho j\rfloor}(\theta+\pi))\Big),
&\text{for $0\le \pcr\le \pi$ and $\theta\in [-2^{-\lfloor \rho j\rfloor}\pi, 2^{-\lfloor \rho j\rfloor}\pi)$;}\\
0, &\text{for all other $\pcr e^{i\theta}$ inside $[-\pi,\pi)^2$.}
\end{cases}
\]
For $\ell=1, \ldots, \mpsi_j$, we define $\fb_{j,\ell}(\pcr e^{i\theta}):=\fb_{j,0}(\pcr e^{i(\theta-2^{\lfloor \rho j\rfloor}\theta \pi(\ell-1)/m)})$ for $0\le \pcr\le \pi$ and $\fb_{j,\ell}(\pcr e^{i\theta}):=0$ for otherwise $\pcr e^{i\theta}\in [-\pi, \pi)^2$ but $\pcr\ge\pi$.
By Corollary~\ref{cor:ntwf:special}, using a similar proof as that of Theorem~\ref{thm:ntwf:special}, we see that all the claims can be easily checked.
\end{proof}

\end{document}